\documentclass[11pt]{amsart}

\usepackage{amsmath}
\usepackage{stmaryrd}
\usepackage{amsfonts}
\usepackage{amssymb}
\usepackage{latexsym}
\usepackage{amscd}
\usepackage{mathrsfs}
\usepackage[all]{xy}

\newdimen\AAdi%
\newbox\AAbo%
\def\AAk#1#2{\s_etbox\AAbo=\hbox{#2}\AAdi=\wd\AAbo\kern#1\AAdi{}}%
\def\AAr#1#2#3{\s_etbox\AAbo=\hbox{#2}\AAdi=\ht\AAbo\raise#1\AAdi\hbox{#3}}%
\font\tenmsb=msbm10 at 12pt \font\sevenmsb=msbm7 at 8pt
\font\fivemsb=msbm5 at 6pt
\newfam\msbfam
\textfont\msbfam=\tenmsb \scriptfont\msbfam=\sevenmsb
\scriptscriptfont\msbfam=\fivemsb

\newtheorem{theorem}{Theorem}

\newtheorem{corollary}[theorem]{Corollary}

\newtheorem{lemma}[theorem]{Lemma}

\numberwithin{equation}{section} \numberwithin{theorem}{section}

\renewcommand{\topmargin}{0cm}
\renewcommand{\oddsidemargin}{5mm}
\renewcommand{\evensidemargin}{5mm}
\renewcommand{\textwidth}{150mm}
\renewcommand{\textheight}{230mm}

\def\C{\mathbb C}
\def\R{\mathbb R}

\def\Z{\mathbb Z}

\def\Z{\mathbb Z}

\def\S{\mathbb S}

\def\na{\nabla}
\def\bn{\overline\nabla}

\def\f#1#2{\frac{#1}{#2}}

\def\a{\alpha}
\def\be{\beta}

\def\r{\Re_{I\!V}}

\def\p#1{\partial #1}

\def\de{\delta}
\def\De{\Delta}
\def\e{\eta}
\def\ep{\epsilon}

\def\G{\Gamma}
\def\g{\gamma}
\def\k{\kappa}
\def\la{\lambda}
\def\La{\Lambda}
\def\lan{\langle}
\def\ran{\rangle}

\def\Om{\Omega}
\def\th{\theta}
\def\Th{\Theta}
\def\si{\sigma}
\def\Si{\Sigma}

\def\r{\rho}
\def\z{\zeta}
\def\div{\mathrm{div}}

\begin{document}

\title
{Minimal cones and self-expanding solutions for mean curvature flows}
\author{Qi Ding}
\address{Shanghai Center for Mathematical Sciences, Fudan University, Shanghai 200433, China}
\email{dingqi@fudan.edu.cn}\email{dingqi09@fudan.edu.cn}

\begin{abstract}
In this paper, we study self-expanding solutions for mean curvature flows and their relationship to minimal cones in Euclidean space. In \cite{Il1}, Ilmanen proved the existence of self-expanding hypersurfaces with prescribed tangent cones at infinity. If the cone is $C^{3,\a}$-regular and mean convex (but not area-minimizing), we can prove that the corresponding self-expanding hypersurfaces are smooth, embedded, and have positive mean curvature everywhere (see Theorem \ref{Intrmcse}).
As a result, for regular minimal but not area-minimizing cones we can give an affirmative answer to a problem arisen by Lawson \cite{Bro}.
\end{abstract}

\maketitle

\section{Introduction}

Minimal cones with isolated singularities in Euclidean space are cones over smooth minimal submanifolds in the unit sphere, whose singularities may or may not vanish after perturbation. In particular, isolated singularities may also exist in general embedded minimal hypersurfaces by Caffarelli-Hardt-Simon \cite{CHS}.
For area-minimizing cones with isolated singularities, there are several ways to perturb away these singularities.
Concerning one-sided perturbations, Hardt-Simon \cite{HaS} showed that there is an oriented connected embedded smooth minimizing hypersurface in every component of $\R^{n+1}\setminus C$, where $C$ is an arbitrary $n$-dimensional minimizing cone in $\R^{n+1}$. Such hypersurfaces form a smooth foliation and are unique if they lie on one side of $C$. After that, Simon-Solomon \cite{SS} showed the uniqueness of minimal hypersurfaces asymptotic to any area-minimizing quadratic minimal cone at infinity in Euclidean space, whereas the uniqueness fails for a large class of strictly minimizing cones showed by C. Chan \cite{Cc}. Concerning two-sided perturbations, the famous minimal graphs discovered by Bombieri-De Giorgi-Giusti in \cite{BDG}, are smooth area-minimizing hypersurfaces which converge to cartesian products of Simons' cones and $\R$ at infinity.
Moreover, R. McIntosh in \cite{M} gave a more general sufficient condition on area-minimizing cones for perturbing away singularities.

In Problem 5.7 of \cite{Bro}, B. Lawson suggested the following problem:

\emph{  Let $C$ be a stable (or minimizing) hypercone in $\R^{n+1}$, and $C_\ep$ be the $\ep$-neighbourhood of $C$ in $\R^{n+1}$. Given $\ep>0$, can one find a smooth embedded hypersurface of positive mean curvature properly embedded in $B_1(0)\cap C_\ep$?}

F.H. Lin first studied this problem and gave a confirmative answer for it when the dimension $n\le7$ in \cite{L1}. Furthermore, let $E$ be a connected component of $\R^{n+1}\setminus C$. If either $C$ is a stable cone with $n=7$ or $C$ is a one-sided area-minimizing cone in $\overline{E}$, Lin \cite{L1} showed the existence of smooth embedded hypersurfaces of positive mean curvature properly embedded in $B_1(0)\cap C_\ep\cap E$ for any small $\ep>0$. An one-sided area minimizing cone may not be area-minimizing, such as a cone over a minimal hypersurface $\S^1\left(\sqrt{\f16}\right)\times\S^5\left(\sqrt{\f56}\right)$ in $\S^7(1)$ proved by Lin \cite{L2}.

A \emph{self-expander} $M$ in Euclidean space is a critical point for the following functional defined on any bounded set $K\subset M$ by
\begin{equation}\label{EK}
\int_K e^{\f{|X|^2}4}d\mu,
\end{equation}
where $d\mu$ is the volume element of $M$ (see \eqref{SE} for the definition). Equivalently, a submanifold $M$ is a self-expander if and only if $\mathcal{M}:\ t\in(0,\infty)\rightarrow\sqrt{t}M$ is a mean curvature flow.
In rough speaking, self-expanders describe the asymptotic behaviour of the longtime solutions for mean curvature flow, and even depict the local structure of the flow after the singularities in a very short time.

Ecker-Huisken \cite{EH1} studied mean curvature flow of entire graphs, and their normalized flows converge to self-expanders under some additional conditions (see also \cite{Sta}). In lecture 2 of \cite{Il1}, Ilmanen showed the existence of $\textbf{E}$-minimizing self-expanding hypersurfaces (see section 2 for the definition of '$\textbf{E}$-minimizing') which converge to prescribed closed cones at infinity in Euclidean space. For 1-dimensional case, self-expanding curves have been used to study planar networks, see \cite{ScSc,INS} for example. Moreover, self-expanders have acted an important role in studying Lagrangian mean curvature flow partially as the absence of nontrivial self-shrinkers for the flow from a zero-Maslov class Lagrangian in $\C^n$ (see \cite{Ne,NT} for instance).

In the current paper, we develop the theory of the self-expanding solutions of mean curvature flows (i.e., self-expanders) and their relationship to minimal cones.
For studying the existence of self-expanders, we first consider the rotational symmetric solutions to graphic self-expanders. With these solutions as barriers, we give another proof of existence of self-expanders for $C^2$-regular cones, which was proved by Ilmanen \cite{Il1} firstly.
Precisely, there exists an $n$-dimensional minimizing current with the weight $e^{\f{|X|^2}4}$ (and with possible singularities) for a prescribed tangent cone over any $C^2$-embedded hypersurface in $\S^{n}$ at infinity (see Theorem \ref{WeakexistSE}). We study the asymptotic decay estimates at infinity for such self-expanders by investigating their Jacobi field operator at infinity.
Furthermore, for each $C^{3,\a}$-regular mean convex but not minimizing cone, we can rule out the singularities of the $\textbf{E}$-minimizing self-expanding current which converges to such cone on one side at infinity.  More precisely, we have the following theorem (see also Theorem \ref{mainTh}).
\begin{theorem}\label{Intrmcse}
Let $C$ be an $n$-dimensional $C^{3,\a}$-regular mean convex but not minimizing cone pointing into the domain $\Om$ with $\p\Om=C$ in $\R^{n+1}$, then there is a unique smooth complete embedded $\textbf{E}$-minimizing self-expanding hypersurface $M$ in $\Om$ with tangent cone $C$ at infinity, where $M$ has positive mean curvature everywhere.
\end{theorem}

Thus we give an affirmative answer to Lawson's problem in \cite{Bro} for regular minimal but not minimizing cones. Moreover, let $M$ be the self-expander in Theorem \ref{Intrmcse}, then $\mathcal{M}:\ t\in(0,\infty)\rightarrow\sqrt{t}M$ is a foliation of $\Om$ moving by mean curvature.
In Theorem \ref{Intrmcse}, the assumption on 'not minimizing' can not be removed. In fact, we can show that if a smooth self-expander $M$ converges to any given regular area-minimizing cone at infinity, then $M$ must be this area-minimizing cone (see Lemma \ref{area-minimSE} for details). Without restriction in $\Om$, uniqueness for self-expanders may fail in view of \cite{AIC,Il1}.

Minimal cones can be seen as special self-expanders, and there are some exact correspondences between them. For a minimal cone $C$  and a domain $\Om$ in Euclidean space with $\p\Om=C$, $C$ is stable if and only if $C$ is a stable self-expander (see Theorem \ref{Stablese}); $C$ is area-minimizing in $\overline{\Om}$ if and only if $C$ is an \textbf{E}-minimizing self-expanding hypersurface in $\overline{\Om}$ (see Theorem \ref{equamcse}). Furthermore, there is an alternative phenomena between the existence of minimal hypersurfaces and self-expanding hypersurfaces with respect to the prescribed minimal cones. More precisely, for an regular nonflat minimal cone $C$, either there exists a complete smooth minimal hypersurface on one side of $C$, but there does not exist any complete smooth self-expander which converges to $C$ at infinity (see Lemma \ref{area-minimSE});
or there exists a complete smooth self-expander which converges to $C$ at infinity, but there does not exist any complete smooth minimal hypersurface on one side of $C$
(see \cite{D0}).

\textbf{Conventions and notation.}
We say an $n$-dimensional \emph{mean convex} hypersurface $\G$ (or an $n$-dimensional hypersurface $\G$ has \emph{positive mean curvature}) pointing into $\Om$ in $\R^{n+1}$, if $\G$ is $C^2$-continuous in $\R^{n+1}$ and the mean curvature of $\G$ is nonnegative (or positive) with respect to the unit normal vector pointing into $\Om$. Note that in this circumstance, we omit "with respect to the unit normal vector" for convenience in this whole paper. Here, the mean curvature is defined by taking trace of \eqref{hij} with the unit normal vector $\nu$ pointing into $\Om$.

Without special illumination, we always think that a cone $C$ has vertex at the origin for convenience in this context.
A \emph{$C^{3,\a}$-regular cone} means the cone over a compact, embedded $C^{3,\a}$-hypersurface without boundary in the sphere for some $\a\in(0,1)$.
We say a cone $C$ \emph{mean convex}, which means that $C\setminus\{0\}$ has nonnegative mean curvature. $c$ denotes a positive constant depending only on the dimension $n$ and the cone $C$ but will be allowed to change from line to line.

\section{Preliminaries and notation}

An $n$-dimensional smooth manifold $M$ is said to be a \emph{self-expander} in $\R^{n+m}$ if it satisfies the elliptic equations
\begin{equation}\label{SE}
H= \frac{X^N}{2},
\end{equation}
where $X$ is the position vector of $M$ in $\R^{n+m}$, $(\cdots)^N$ is the projection into the normal bundle of $M$, and $H$ is the mean curvature vector of $M$. The reason of calling is that $\sqrt{t}M$ satisfies the mean curvature flow for $t>0$. In fact, let $M_t=\sqrt{t}M$ and $F_t=\sqrt{t}X=(\sqrt{t}x_1,\cdots,\sqrt{t}x_{m+n})$, then
$$\f{dF_t}{dt}=\f1{2\sqrt{t}}X=\f1{2t}F_t.$$
Hence
\begin{equation}
\left(\f{dF_t}{dt}\right)^N=\f1{2t}F_t^N=H_{M_t},
\end{equation}
where $N$ also denotes the projection into the normal bundle of $M_t$, $H_{M_t}$ is the mean curvature vector of $M_t$. Conversely, if $\mathcal{M}:\ t\in(0,\infty)\rightarrow\sqrt{t}M$ is a mean curvature flow, then obviously $M$ is a self-expander by the above argument.

Let $\na$ and $\De$ be the Levi-Civita connection and Laplacian of self-expander $M$, respectively.
For any $X=(x_1,\cdots,x_{n+m})\in\R^{n+m}$, we have
$$\De X=H,$$
and
\begin{equation}\aligned\label{laplace}
\De|X|^2=2\lan X,\De X\ran+2|\na X|^2=2\lan X,H\ran+2n=|X^N|^2+2n=2n+4|H|^2.
\endaligned
\end{equation}

If $M$ can be written as a graph over $\Om\subset\R^n$ with the graphic function $u$ in $\R^{n+1}$, namely, $M=\{(x,u(x))\in\R^{n+1}|\ x\in\Om\}$, then \eqref{SE} implies
\begin{equation}\aligned\label{Graphv}
\div\left(\f{Du}{\sqrt{1+|Du|^2}}\right)=\f{-x_iu_i+u}{2\sqrt{1+|Du|^2}},
\endaligned
\end{equation}
where 'div' is the divergence on $\R^n$. Moreover, \eqref{Graphv} is equivalent to
\begin{equation}\aligned\label{Graphu}
g^{ij}u_{ij}=\f{-x_iu_i+u}2,
\endaligned
\end{equation}
where $(g^{ij})$ is the inverse matrix of $(g_{ij})$ with $g_{ij}=\de_{ij}+u_iu_j$.

Let $\la$ be a function on $\R^+$ satisfying
$$re^{\f{r^2}{4n}}=\la\left(\int_0^re^{\f{t^2}{4n}}dt\right).$$
Then the weighted space $\left(\R^{n+m},e^{\f{|X|^2}{2n}}\sum_{i=1}^{n+m}dx_i^2\right)$ can be written in a polar coordinate as $N\triangleq(\R^{n+m},\mathbf{g})$ with
\begin{equation*}\aligned
\mathbf{g}=&e^{\f{r^2}{2n}}\left(dr^2+r^2\si_{_{\S^{n+m-1}}}\right)=\left(d\int_0^re^{\f{t^2}{4n}}dt\right)^2+\la^2\left(\int_0^re^{\f{t^2}{4n}}dt\right)
\si_{_{\S^{n+m-1}}}\\
=&d\r^2+\la^2(\r)\si_{_{\S^{n+m-1}}},
\endaligned
\end{equation*}
where $\si_{_{\S^{n+m-1}}}$ is the standard metric of $(n+m-1)$-dimensional unit sphere in $\R^{n+m}$.
In particular, $N$ has non-positive sectional curvature by \cite{Lip} or (3.30) in \cite{DJX}. Each $n$-dimensional self-expander in $\R^{n+m}$ is equivalent to an $n$-dimensional minimal submanifold in $N$.

By Rademacher's theorem, a locally $n$-rectifiable set has tangent spaces at almost every point. Let $G(n,p)$ be the space of (unoriented) $n$-planes through the origin in $\R^p$.
A $n(n<p)$-rectifiable varifold in $\R^p$ corresponds to an $n$-varifold defined by Radon measure on $G(n,p)$ as Chapter 38 in \cite{S}.
We call a locally $n$-rectifiable varifold $S$ in $\R^p$ \emph{an $n$-varifold self-expander} if 
for any $C^1$-vector field $Y$ with compact support in spt$S$, we have
\begin{equation}\aligned\label{WeakSE}
\int_{\R^p\times G(n,p)}\overline{\div}_\omega YdS(X,\omega)=-\f12\int_{\R^p\times G(n,p)}\lan X^N,Y\ran dS(X,\omega),
\endaligned
\end{equation}
where $\overline{\div}$ represents the divergence on $\R^p$, and $\omega$ is an $n$-dimensional tangent plane of $S$ at the considered point if it exists.
Then $\f12X^N$ is the generalized mean curvature of $S$ (see \cite{LY}\cite{S} for its definition). \emph{An integer $n$-varifold self-expander} is an abbreviation of an integer multiplicity $n$-varifold self-expander. Sometimes, we omit '$n$' when we don't emphasize the dimension of '$n$-varifold self-expander'.

From monotonicity identity (formula 17.3 in \cite{S}), we have
\begin{equation}\aligned\label{monSEr-nBr}
\f{d}{d\r}\left(\r^{-n}\mathcal{H}^n(S\cap B_\r)\right)=\f{d}{d\r}\int_{S\cap B_\r}\f{|X^N|^2}{|X|^{n+2}}d\mu_S+\f12\r^{-n-1}\int_{S\cap B_\r}\left|X^N\right|^2d\mu_S,
\endaligned
\end{equation}
where $B_\r$ denotes the ball in $\R^{n+1}$ with radius $\r$ and centered at the origin, $\mathcal{H}^n(K)$ denotes the $n$-dimensional Hausdorff measure of any set $K\subset\R^{n+1}$, and $d\mu_S$ is the Radon measure corresponding to the varifold $S$.

Now let us recall some classical definitions of currents (see \cite{LY}\cite{S} for example) and define a weighted mass for rectifiable currents.
Let $U$ be an open subset of $\R^p$ for $p>n$, $I_{n,p}=\{\a=(i_1,\cdots,i_n)\in\Z_+^n|\ 1\le i_1<\cdots<i_n\le p\}$. Denote $E^n(U)$ be the set including all smooth $n$-forms $\omega=\sum_{\a\in I_{n,p}}a_\a dx^\a$, where $a_\a\in C^\infty(U)$ and $dx^\a=dx^{i_1}\wedge\cdots\wedge dx^{i_n}$ if $\a=(i_1,\cdots,i_n)\in I_{n,p}$. Let $\mathcal{D}^n(U)$ denote the set of $\omega=\sum_{\a\in I_{n,p}}a_\a dx^\a\in E^n(U)$ such that each $a_\a$ has compact support in $U$. For any $\omega\in\mathcal{D}^n(U)$ one denotes a norm $|\cdot|_U$ by
$$|\omega|_U=\sup_{x\in U}\lan\omega(x),\omega(x)\ran^{\f12}.$$
Denote $\mathcal{D}_n(U)$ be the set of $n$-currents in $U$, which are continuous linear functionals on $\mathcal{D}^n(U)$.
For each $T\in\mathcal{D}_n(U)$ and each open set $W$ in $U$, one defines
$$\mathbf{M}_W(T)=\sup_{|\omega|_U\le1,\omega\in\mathcal{D}^n(U),\mathrm{spt}\omega\subset W}T(\omega),$$
and the mass of $T$ by
$$\mathbf{M}(T)=\sup_{|\omega|_U\le1,\omega\in\mathcal{D}^n(U)}T(\omega).$$

If $\mathbf{M}_W(T)<\infty$ for $T\in\mathcal{D}_n(U)$ and $W\Subset U$, then by the Riesz Representation, there is a Radon measure $\mu_T$ on $U$ and $\mu_T$-measurable vector-valued function $\overrightarrow{T}$ with values in the spaces of $n$-vectors $\La_n(\R^p)$, $|\overrightarrow{T}|=1$ $\mu_T$-a.e., such that
$$T(\omega)=\int\lan\omega(x),\overrightarrow{T}(x)\ran d\mu_T(x).$$
Let $X$ be the position vector in $\R^p$.
We define a new mass in $W$ with the weight $e^{\f{|X|^2}4}$ by
$$\textbf{E}_W(T)=\sup_{|\omega|_U\le1,\omega\in\mathcal{D}^n(U),\mathrm{spt}\omega\subset W}\int\lan\omega(X),\overrightarrow{T}(X)\ran e^{\f{|X|^2}4}d\mu_T(X).$$

We call $T$ \emph{an \textbf{E}-minimizing self-expanding current} in $W\subset U$, if $T$ is an integer multiplicity locally rectifiable current in a subset $W$ of $U$, and $\textbf{E}_K(T)\le\textbf{E}_K(T')$
whenever $K\Subset U$, $\p T'=\p T$, and spt$(T'-T)$ is a compact subset of $K\cap W$. And we call $\mathrm{spt}T$ \emph{an \textbf{E}-minimizing self-expanding hypersurface(curve)} in $W$ if $p=n+1(n=1)$.

Let $E$ be a Lebesgue measurable set in $\R^{p}$, we denote
$$s\left(E+\xi\right)=\{s(X+\xi)|\ X\in E\}$$
for any $s>0$ and $\xi\in\R^p$, and
$\llbracket E\rrbracket$ be the multiplicity one current associated with $E$. For an $n$-dimensional integral current $T$ in $\R^p$, we denote $|T|$ be the integral varifold associated with $T$. Let $\e_{\xi,s}$ be a translation plus homothety given by
$$\e_{\xi,s}(X)=s^{-1}(X-\xi).$$
Let $S$ be a varifold or current in $\R^{p}$, we define $s^{-1}\left(S-\xi\right)=\e_{\xi,s\sharp} S$ in this text for convenience, where the definition of $\e_{\xi,s\sharp}$ is
the same as in \cite{HaS}(or \cite{S}). We will use another way to describe the multiplicity of $S$.

\section{Geometry of self-expanders}

\subsection{\textbf{E}-minimizing graphic self-expanding hypersurfaces}
Every self-expander with codimension 1 is an \textbf{E}-minimizing self-expanding hypersurface in $\overline{\Om}\times\R$ if it is a graph over a domain $\Om\subset\R^n$.
\begin{lemma}\label{comp}
Let $\Om$ be a bounded domain in $\R^n$ and $M$ be a graphic self-expander in
$\overline{\Om}\times\R$ with the graphic function $u$ and with the volume element $d\mu_M$. For any smooth compact hypersurface $W\subset\overline{\Om}\times\R$ with $\p M=\p W$ and the volume element $d\mu_W$, one has
\begin{equation}\aligned
\int_{M}e^{\f{|X|^2}4}d\mu_M\le\int_{W}e^{\f{|X|^2}4}d\mu_{W},
\endaligned
\end{equation}
where the above inequality attains equality if and only if $W=M$.
\end{lemma}
\begin{proof}
Let $U$ be the domain in $N$ enclosed by $M$ and $W$. Without loss of generality, we assume that $\overline{M}\cap \overline{W}=\p M$.
Let $Y$ be a vector field on $M$ defined by
$$Y=-\sum_{i=1}^n\f{u_i}{\sqrt{1+|Du|^2}}e^{\f{\sum_{i=1}^{n+1}x_i^2}4}E_i+\f{1}{\sqrt{1+|Du|^2}}e^{\f{\sum_{i=1}^{n+1}x_i^2}4}E_{n+1}.$$
Viewing $u_i$ and $|Du|$ as functions on $\Om$ and translating $Y$ to $W$ along the $E_{n+1}$ ($x_{n+1}$ axis) direction, we obtain a vector field on $U$, denoted by $Y$, as well.
From (\ref{Graphv}) we have
\begin{equation*}\aligned
\overline{\div}(Y)=&-\sum_{i=1}^n\left(\p_{x_i}\left(\f{u_i}{\sqrt{1+|Du|^2}}\right)+\f{x_iu_i}{2\sqrt{1+|Du|^2}}\right)e^{\f{|X|^2}4}+\f{x_{n+1}}{2\sqrt{1+|Du|^2}}e^{\f{|X|^2}4}\\
=&\f{x_{n+1}-u}{2\sqrt{1+|Du|^2}}e^{\f{|X|^2}4},
\endaligned
\end{equation*}
where $\overline{\div}$ stands for the divergence on $\R^{n+1}$, and $|X|^2=\sum_{i=1}^{n+1}x_i^2$.
Let $\nu_M,\nu_W$ be the unit normal vectors of $M,W$ respectively, such that $Y|_M=e^{\f{|X|^2}4}\nu_M$.
If $W$ is above $M$, by Stokes' theorem, up to a sign of $\nu_W$ we get
\begin{equation}\aligned\label{Gaussf}
0\le\int_{U}\f{x_{n+1}-u}{2\sqrt{1+|Du|^2}}e^{\f{|X|^2}4}=&\int_{U}\overline{\div}(Y)=-\int_{M}\lan Y,\nu_M\ran d\mu_M+\int_{W}\lan Y,\nu_W\ran d\mu_{W}\\
\le&-\int_{M}e^{\f{|X|^2}4} d\mu_M+\int_{W}e^{\f{|X|^2}4}d\mu_{W},
\endaligned
\end{equation}
where equality holds if and only if $M=W$.
It is similar for the case that $W$ is below $M$, and we complete the proof.
\end{proof}

\subsection{Bochner type formula for second fundamental form}
For a hypersurface $M$ in $\R^{n+1}$, we choose a local orthonormal frame field $\{e_1,\cdots, e_n\}$ of $M$ at any considered point and let $\overline{\na}$ be the Levi-Civita connection of $\R^{n+1}$. Let $A$ be the second fundamental form and $A_{e_i e_j}=h(e_i,e_j)\nu=h_{ij}\nu$. Here, $\nu$ is the unit normal vector of $M$ in $\R^{n+1}$. Then the coefficients of the second fundamental form $h_{ij}$ are a symmetric $2-$tensor on $M$ and
\begin{equation}\aligned\label{hij}
h_{ij}=\lan\overline{\na}_{e_i}e_j,\nu\ran.
\endaligned
\end{equation}
Denote mean curvature $H=\sum_ih_{ii}$ and the square of the second fundamental form $|A|^2=\sum_{i,j}h_{ij}^2$.
\begin{lemma}\label{LhijB}
Let $M$ be a self-expander in $\R^{n+1}$. Then
\begin{equation}\aligned\label{Lh}
\De h_{ij}+\f12\lan X,\na h_{ij}\ran+\left(\f12+|A|^2\right)h_{ij}=0.
\endaligned
\end{equation}
\end{lemma}
\begin{proof}
By Ricci identity, one has
\begin{equation}\aligned
\De h_{ij}=h_{ijkk}=h_{ikjk}=h_{ikkj}+h_{il}R_{klkj}+h_{kl}R_{ilkj}.
\endaligned
\end{equation}
Combining Gauss formula $R_{ijkl}=h_{ik}h_{jl}-h_{il}h_{jk}$ and \eqref{SE}, we have
\begin{equation}\aligned\label{Deh}
\De h_{ij}=&H_{ij}+h_{il}(h_{kk}h_{lj}-h_{kj}h_{kl})+h_{kl}(h_{ik}h_{lj}-h_{ij}h_{kl})\\
=&\f12(\lan X,\nu\ran)_{ij}+Hh_{ik}h_{jk}-|A|^2h_{ij}.
\endaligned
\end{equation}
Since
\begin{equation}\aligned
\na_{e_i}(h(e_j,e_k))=&(\na_{e_i}h)(e_j,e_k)+h(\na_{e_i}e_j,e_k)+h(e_j,\na_{e_i}e_k)\\
=&h_{jki}+\lan\na_{e_i}e_j,e_l\ran h_{kl}+\lan\na_{e_i}e_k,e_l\ran h_{jl},
\endaligned
\end{equation}
then
\begin{equation}\aligned\label{dH}
(\lan X,\nu\ran)_{ij}=&\na_{e_i}\na_{e_j}\lan X,\nu\ran-\na_{\na_{e_i}{e_j}}\lan X,\nu\ran\\
=&-\na_{e_i}(\lan X,e_k\ran h_{jk})-\lan\na_{e_i}e_j,e_k\ran\lan X,\overline{\na}_{e_k}\nu\ran\\
=&-h_{ij}-\lan X,\overline{\na}_{e_i}e_k\ran h_{jk}-\lan X,e_k\ran\big(h_{jki}+\lan\na_{e_i}e_j,e_l\ran h_{kl}\\&+\lan\na_{e_i}e_k,e_l\ran h_{jl}\big)+\lan\na_{e_i}e_j,e_k\ran\lan X,{e_l}\ran h_{kl}\\
=&-h_{ij}-\lan X,\nu\ran h_{ik}h_{jk}-\lan X,\na_{e_i}e_k\ran h_{jk}-\lan X,e_k\ran h_{ijk}\\
&-\lan X,e_k\ran\lan\na_{e_i}e_k,e_l\ran h_{jl}\\
=&-h_{ij}-\lan X,\nu\ran h_{ik}h_{jk}-\lan X,e_l\ran\lan e_l,\na_{e_i}e_k\ran h_{jk}-\lan X,\na h_{ij}\ran\\
&+\lan X,e_k\ran\lan e_k,\na_{e_i}e_l\ran h_{jl}\\
=&-h_{ij}-2H h_{ik}h_{jk}-\lan X,\na h_{ij}\ran.
\endaligned
\end{equation}
Combining \eqref{Deh} and \eqref{dH}, we complete the Lemma.
\end{proof}
Remark. Lemma \ref{Lh} can also follow directly from the evolution equation for $h_{ij}$ under the mean curvature flow $\sqrt{t}M$, as the reviewer pointed out.

Taking the trace of \eqref{Lh}, we obtain
\begin{equation}\aligned\label{meancurvature}
\De H+\f12\lan X,\na H\ran+\left(\f12+|A|^2\right)H=0.
\endaligned
\end{equation}
Analog to \eqref{dH}, we have
\begin{equation}\aligned
(\lan E_{n+1},\nu\ran)_{ij}=&\na_{e_i}\na_{e_j}\lan E_{n+1},\nu\ran-\na_{\na_{e_i}{e_j}}\lan E_{n+1},\nu\ran\\
=&-\na_{e_i}(\lan E_{n+1},e_k\ran h_{jk})-\lan\na_{e_i}e_j,e_k\ran\lan E_{n+1},\overline{\na}_{e_k}\nu\ran\\
=&-\lan E_{n+1},\overline{\na}_{e_i}e_k\ran h_{jk}-\lan E_{n+1},e_k\ran\big(h_{jki}+\lan\na_{e_i}e_j,e_l\ran h_{kl}\\&+\lan\na_{e_i}e_k,e_l\ran h_{jl}\big)+\lan\na_{e_i}e_j,e_k\ran\lan E_{n+1},{e_l}\ran h_{kl}\\
=&-\lan E_{n+1},\nu\ran h_{ik}h_{jk}-\lan E_{n+1},\na_{e_i}e_k\ran h_{jk}-\lan E_{n+1},e_k\ran h_{ijk}\\
&-\lan E_{n+1},e_k\ran\lan\na_{e_i}e_k,e_l\ran h_{jl}\\
=&-\lan E_{n+1},\nu\ran h_{ik}h_{jk}-\lan E_{n+1},e_l\ran\lan e_l,\na_{e_i}e_k\ran h_{jk}-\lan E_{n+1},\na h_{ij}\ran\\
&+\lan E_{n+1},e_k\ran\lan e_k,\na_{e_i}e_l\ran h_{jl}\\
=&-\lan E_{n+1},\nu\ran h_{ik}h_{jk}-\lan E_{n+1},\na h_{ij}\ran.
\endaligned
\end{equation}
Hence
\begin{equation}\aligned
\De\lan E_{n+1},\nu\ran=&-|A|^2\lan E_{n+1},\nu\ran -\lan E_{n+1},\na H\ran\\
=&-|A|^2\lan E_{n+1},\nu\ran -\f12\lan E_{n+1},e_i\ran\lan X,\na_{e_i}\nu\ran\\
=&-|A|^2\lan E_{n+1},\nu\ran -\f12\lan X,\na\lan E_{n+1},\nu\ran\ran.
\endaligned
\end{equation}
When $M$ is a graph with a graphic function $u$, we define $w=-\log \lan E_{n+1},\nu\ran=\f12\log(1+|Du|^2)$. Then
\begin{equation}\aligned\label{Dewgra}
\De w=|A|^2-\f12\lan X,\na w\ran+|\na w|^2.
\endaligned
\end{equation}

\subsection{Stable self-expanders}
For a $C^2$-hypersurface $M\subset\R^{n+1}$, we define a family of hypersurfaces $M_s$ by $X(\cdot,s):\ M\rightarrow\R^{n+1}$ with $X(p,s)=p+sf(p)\nu(p)$, where $\nu$ is the unit normal vector of $M$, and $f\in C^\infty_c(M)$. Let $\nu(p,s)$ denote the unit normal to $M_s$ at the point $X(p,s)$ and $H(p,s)$ be the mean curvature of $M_s$ at the point $X(p,s)$.
So we get
\begin{equation}\aligned
\f{\p}{\p s}\int_{M_s}e^{\f{|X|^2}4}d\mu=\int_{M_s}\left(-H+\f{\lan X,\nu\ran}2\right)fe^{\f{|X|^2}4}d\mu.
\endaligned
\end{equation}
Hence self-expanders are critical points of the area-functional with the weight $e^{\f{|X|^2}4}$.
By standard calculations on $H$ and $\nu$, we have (see Theorem 3.2 in \cite{HP} for example)
\begin{equation}\aligned\label{Hnusp0}
\f{\p H}{\p s}(p,0)&=\De_M f(p)+|A|^2(p)f(p),\\
\f{\p \nu}{\p s}(p,0)&=-\na_M f(p).
\endaligned
\end{equation}
Then we get
\begin{equation}\aligned
\f{\p^2}{\p s^2}\bigg|_{s=0}\int_{M_s}e^{\f{|X|^2}4}d\mu=\int_{M}\left(-\De_M f-|A|^2f-\f{\lan X,\na_M f\ran}2+\f f2\right)fe^{\f{|X|^2}4}d\mu.
\endaligned
\end{equation}
We say that a smooth self-expander $M$ is \emph{stable} if
\begin{equation}\aligned\label{SE2var}
\int_{M}\left(|\na_M f|^2-|A|^2f^2+\f12f^2\right)e^{\f{|X|^2}4}d\mu\ge0.
\endaligned
\end{equation}
for any $f\in C^\infty_c(M)$.

We say an $n$-dimensional minimal cone $C$ in $\R^{n+1}(n\ge2)$ \emph{stable}, if the second variation of area functional on $C\setminus\{0\}$ is nonnegative. We say $C$ a \emph{stable self-expander} if $C\setminus\{0\}$ is a stable self-expander.
\begin{theorem}\label{Stablese}
Any minimal cone with an isolated singularity is stable if and only if it is a stable self-expander.
\end{theorem}
\begin{proof}
Let $\Si$ be a minimal hypersurface in $\S^n$ with the second fundament form $A_\Si$. As every 1 dimensional smooth minimal surface in $\S^2$ is an equator, we assume $n\ge3$.  With respect to \eqref{SE2var}, for any $f=f(\xi,t)\in C^\infty_c(\Si\times (0,\infty))$ we define
\begin{equation}\aligned
I(f)\triangleq\int_{\Si\times\R^+}\left(-\De_\Si f-|A_\Si|^2f-(n-1)t\f{\p f}{\p t}-t^2\f{\p^2 f}{\p t^2}-\f {t^3}2\f{\p f}{\p t}+\f{t^2}2 f\right)t^{n-3}fe^{\f{t^2}4}d\mu_\Si dt.
\endaligned
\end{equation}
Integrating by parts implies
\begin{equation}\aligned\label{If***}
I(f)=\int_{\Si\times\R^+}\left(-f\De_\Si f-|A_\Si|^2f^2+t^2\left(\f{\p f}{\p t}\right)^2+\f{t^2}2 f^2\right)t^{n-3}e^{\f{t^2}4}d\mu_\Si dt.
\endaligned
\end{equation}
Since
\begin{equation}\aligned
&\int_{0}^\infty\left(\f{\p}{\p t}\left(f e^{\f{t^2}8}\right)\right)^2t^{n-1}e^{\f{t^2}4}dt=\int_{0}^\infty\left(\left(\f{\p f}{\p t}\right)^2+\f{t}2f\f{\p f}{\p t}+\f{t^2}{16}f^2\right)t^{n-1}e^{\f{t^2}4}dt\\
&=\int_{0}^\infty\left(\f{\p f}{\p t}\right)^2t^{n-1}e^{\f{t^2}4}dt-\int_0^\infty\left(\f n4+\f{t^2}{16}\right)f^2t^{n-1}e^{\f{t^2}4}dt,
\endaligned
\end{equation}
then from \eqref{If***} one has
\begin{equation}\aligned
I(f)\ge\int_{\Si\times\R^+}\left(-fe^{\f{t^2}8}\De_\Si\left(fe^{\f{t^2}8}\right)-|A_\Si|^2\left(f e^{\f{t^2}8}\right)^2+t^2\left(\f{\p}{\p t}\left(f e^{\f{t^2}8}\right)\right)^2\right)t^{n-3} d\mu_\Si dt.
\endaligned
\end{equation}
If $C\Si$ is a stable minimal cone, then $I(f)\ge0$ for any $f=f(\xi,t)\in C^\infty_c(\Si\times (0,\infty))$ by the definition of stable minimal cones.
Hence we conclude that $C\Si$ is a stable self-expander.

Let $\la_1$ be the first eigenvalue of the operator $\De_\Si+|A_\Si|^2$ with the corresponding eigenfunction $\varphi_1$ on $\Si$ satisfying $\int_\Si \varphi_1^2=1$.
Let $\e\in C^\infty_c((0,\infty))$ and $f(y,t)=\varphi_1(y)\e(t)$ for $(y,t)\in\Si\times(0,\infty)$, then from \eqref{If***}
\begin{equation}\aligned\label{IJfphi}
I(f)=&\int_{\Si\times\R^+}\left(\la_1\e^2+\f12t^2\e^2+t^2\left(\f{\p \e}{\p t}\right)^2\right)t^{n-3}\varphi_1^2 e^{\f{t^2}4}d\mu_\Si dt\\
=&\int_{0}^\infty\left(\la_1\e^2+\f12t^2\e^2+t^2\left(\f{\p \e}{\p t}\right)^2\right)t^{n-3} e^{\f{t^2}4}dt\triangleq I_0(\e),
\endaligned
\end{equation}
where we have used $\De_\Si\varphi_1+|A_\Si|^2\varphi_1+\la_1\varphi_1=0$ in the first step of \eqref{IJfphi}.

If $C\Si$ is a unstable minimal cone, then by scaling (see the structure of the formula (6.4.3) in \cite{X}) there is a function $f=f(\xi,t)\in C^\infty_c(\Si\times [\ep_0,1])$ for some $0<\ep_0<1$ such that
\begin{equation}\aligned
\int_{\Si\times[\ep_0,1]}\left(-\De_\Si f-|A_\Si|^2f-(n-1)t\f{\p f}{\p t}-t^2\f{\p^2 f}{\p t^2}\right)t^{n-3}f d\mu_\Si dt<0.
\endaligned
\end{equation}
All the eigenvalues of Jacobi operator for $C\Si$ on $\Si\times [\ep_0,1]$ can be computed from \cite{Si} or Section 6.4 in \cite{X}.
Hence, with Lemma 6.4.5 in \cite{X}, the first eigenvalue $\la_1+\f{(n-2)^2}4+\left(\f{\pi}{\log\ep_0}\right)^2<0$. For any $\ep>0$, we define a Lipschitz function
\begin{equation*}
\e_{\de,R}(t)=\left\{\begin{split}
2\de^{\ep-\f n2}\left(e^{-\f{\de^2}4}-e^{-\f{R^2}4}\right)\left(t-\f\de2\right)&,\qquad \left[\f\de2,\de\right)\\
t^{\ep+1-\f n2}\left(e^{-\f{t^2}4}-e^{-\f{R^2}4}\right)&,\qquad [\de,R)\\
0\qquad&,\qquad \left(0,\f\de2\right)\cup[R,\infty)
\end{split}\right.
\end{equation*}
for any $R>1>\de>0$. By the definition of the functional $I_0$ in \eqref{IJfphi},
\begin{equation}\aligned
I_0(\e_{\de,R})=&\int_{\f\de2}^\de\left(\la_1+\f12t^2+\f{t^2}{(t-\f\de2)^2}\right)\e_{\de,R}^2t^{n-3} e^{\f{t^2}4}dt\\
&+\int_{\de}^R\left(\la_1\e_{\de,R}^2+\f12t^2\e_{\de,R}^2+t^2\left(\f{\p \e_{\de,R}}{\p t}\right)^2\right)t^{n-3} e^{\f{t^2}4}dt.
\endaligned
\end{equation}
Set $\e(t)=t^{\ep+1-\f n2}e^{-\f{t^2}4}$ on $(0,\infty)$. Since
\begin{equation}\aligned
&\left|\int_{\f\de2}^\de\left(\la_1+\f12t^2+\f{t^2}{(t-\f\de2)^2}\right)\e_{\de,R}^2t^{n-3} e^{\f{t^2}4}dt\right|\\
\le&\int_{\f\de2}^\de\left|\la_1+\f12t^2+\f{t^2}{(t-\f\de2)^2}\right|4\de^{2\ep-n}\left(t-\f\de2\right)^2t^{n-3}dt\\
\le&4\left(|\la_1|+\f32\right)\int_{\f\de2}^\de\de^{2\ep-n}t^{n-1} dt\le\left(|\la_1|+\f32\right)\f{4}n\de^{2\ep},
\endaligned
\end{equation}
then
\begin{equation}\aligned
&\lim_{R\rightarrow\infty,\de\rightarrow0}I_0(\e_{\de,R})=\int_{0}^\infty\left(\la_1\e^2+\f12t^2\e^2+t^2\left(\f{\p \e}{\p t}\right)^2\right)t^{n-3} e^{\f{t^2}4}dt\\
=&\int_0^\infty\left(\la_1t^{2\ep-1}+\f12t^{2\ep+1}+\left(\ep+1-\f n2-\f{t^2}2\right)^2t^{2\ep-1}\right)e^{-\f{t^2}4}dt\\
=&2^{2\ep-1}\int_0^\infty\left(\la_1+2s+\left(\ep+1-\f n2-2s\right)^2\right)s^{\ep-1}e^{-s}ds\\
=&2^{2\ep-1}\left(\left(\la_1+\left(\ep+1-\f n2\right)^2\right)\G(\ep)+2(n-1-2\ep)\G(1+\ep)+4\G(2+\ep)\right),
\endaligned
\end{equation}
where $\G(s)$ is the standard Gamma function $\int_0^\infty t^{s-1}e^{-t}dt$ for $s>0$.
Note $\la_1\le-\f{(n-2)^2}4-\left(\f{\pi}{\log\ep_0}\right)^2$ for $n\ge3$ and $\lim_{\ep\rightarrow0^+}\G(\ep)=+\infty$, then for any sufficiently small constant $\ep>0$, we have
\begin{equation}\aligned
&\lim_{R\rightarrow\infty,\de\rightarrow0}I_0(\e_{\de,R})<0.
\endaligned
\end{equation}
Hence $C\Si$ is a unstable self-expander. We complete the proof.
\end{proof}

\section{Rotational graphic self-expanders}
We want to study the existence of self-expanders via constructing barrier functions. Rotational symmetric self-expanders are a class of simple but important self-expanders in Euclidean space.
In this setting, \eqref{Graphv} reduces to the following ODE
\begin{equation}\aligned\label{Ju}
\mathcal {J}u\triangleq\f{u_{rr}}{1+u_r^2}+\f{n-1}ru_r+\f12ru_r-\f12u=0\qquad \mathrm{on}\ (0,\infty)
\endaligned
\end{equation}
with $n\ge2$, $u'(0)=0$ and $\lim_{r\rightarrow\infty}\f{u(r)}r=\k>0$.

Set
$$w(r)=\k r+\f{K}r$$
for $r>0$ and some constant $K$ to be defined. Clearly, $w_r=\k-\f{K}{r^2}$ and $w_{rr}=\f{2K}{r^3}$. Then
\begin{equation}\aligned\label{Jw***}
\mathcal {J}w=&\f{2K}{r^3(1+w_r^2)}+\f{(n-1)\k}r-\f{(n-1)K}{r^3}+\f12\k r-\f{K}{2r}-\f12\k r-\f{K}{2r}\\
=&\f{2K}{r^3(1+w_r^2)}+\f{(n-1)\k-K}r-\f{(n-1)K}{r^3}.
\endaligned
\end{equation}
For $n\ge2$ and $R\ge n$, put $K^\pm_R=\pm(1+\k)R^2$. Then $|\k-K_R^\pm r^{-2}|\ge1$ and
\begin{equation}\aligned\label{Jw***}
\mathcal {J}\left(\k r+\f{K^+_R}r-\f{K^+_R}R\right)\le\f{K^+_R}{r^3}+\f{(n-1)\k-K^+_R}r-\f{(n-1)K^+_R}{r^3}+\f{K^+_R}{2R}<0
\endaligned
\end{equation}
on $(0,R)$. Similarly,
\begin{equation}\aligned\label{Jw***}
\mathcal {J}\left(\k r+\f{K^-_R}r-\f{K^-_R}R\right)\ge\f{K^-_R}{r^3}+\f{(n-1)\k-K^-_R}r-\f{(n-1)K^-_R}{r^3}+\f{K^-_R}{2R}>0
\endaligned
\end{equation}
on $(0,R)$. Namely, $\k |x|+\f{K^\pm_R}{|x|}-\f{K^\pm_R}R$ are super(sub)solutions to \eqref{Graphv}.
Denote $B_R$ be the ball centered at the origin with radius $R$ in $\R^{n+1}$.
Hence, the smooth solution to \eqref{Graphv} with boundary data $\k R>0$ on $\p B_R$ for $R\ge n$ has \emph{a priori} boundary gradient estimate, which is also a global gradient estimate from the maximum principle for \eqref{Dewgra}.
From the existence theory of elliptic equations by the continuity method (as well as Schauder estimates),
there exists a unique smooth solution $u_{\k,R}$ to \eqref{Graphv} with $u_{\k,R}=\k R>0$ on $\p B_R$ for each $R\ge n$.
Let $\mathscr{A}$ denote an $n\times n$-orthonormal matrix with $\mathscr{A}\mathscr{A}^T=I$, then it is easy to see that $u_{\k,R}(\mathscr{A}\cdot)$ is also a smooth solution to \eqref{Graphv} with $u_{\k,R}(\mathscr{A}\cdot)=\k R$ on $\p B_R$. By the uniqueness of $u_{\k,R}$, we conclude that $u_{\k,R}$ is a rotationally symmetric solution.
Namely, there exists a unique solution $\phi_{\k,R}$ to \eqref{Ju} on $(0,R)$ with $\phi_{\k,R}(R)=\k R>0$ and $\phi_{\k,R}'(0)=0$.
If there is a point $r_*\in(0,R)$ with $\phi_{\k,R}'(r_*)<0$,
then there are a hyperplane $\{x_{n+1}=r^*\}$ and a domain $\Om_*$ with $\overline{\Om_*}\subset B_R$ such that $|u_{\k,R}|\ge |r^*|$ in $\Om_*$ and $u_{\k,R}= r^*$ on $\p\Om_*$.
However, the rectifiable hypersurface $\{(x,u_{\k,R})|\ x\in B_R\setminus\Om_*\}\cup(\Om_*\times\{r^*\})$ has smaller mass with the weight $e^{\f{|X|^2}4}$ than $\{(x,\phi_{\k,R}(|x|))|\ x\in B_R\}$,
which contradicts to that graph$_{u_{\k,R}}=\{(x,\phi_{\k,R}(|x|))|\ x\in B_R\}$ is a smooth \textbf{E}-minimizing self-expanding hypersurface in $B_R\times\R$ from Lemma \ref{comp}.
So we conclude that $\phi_{\k,R}'\ge0$ on $[0,R)$.
Moreover, by comparing the mass of the rectifiable hypersurface $\{(x,\max\{0,u_{\k,R}(|x|)\})|\ x\in B_R\}$ with the weight $e^{\f{|X|^2}4}$,
we concluding that $\phi_{\k,R}(0)\ge0$ again from the smooth \textbf{E}-minimizing self-expanding hypersurface graph$_{u_{\k,R}}$.

Now let us give a refined super(sub)solutions to \eqref{Ju}.
When $n=2$, we set $K=2+2\k$. For $r\in(0,\sqrt{2}]$,
$$w_r^2=\left(\f{K}{r^2}-\k\right)^2\ge\left(\f{K}{2}-\k\right)^2=1,$$
then obviously $\mathcal{J}w\le0$. For $r\in[\sqrt{2},\infty)$,
$$\mathcal{J}w\le\f{2K}{r^3}+\f{\k-K}r-\f{K}{r^3}=\f{K}{r^3}+\f{\k-K}r\le\f{K}{2r}+\f{\k-K}r=-\f1r<0.$$
Hence we always have $\mathcal{J}w\le0$ on $(0,\infty)$. When $n\ge3$, $\mathcal{J}w\le0$ holds clearly if $K=(n-1)\k$. Namely,
\begin{equation}\aligned\label{Jkn-1k}
\mathcal{J}\left(\k r+\f{(n-1)\k}r\right)\le0\qquad \mathrm{for}\ n\ge3.
\endaligned
\end{equation}
Note that $\mathcal{J}(\k r)\ge0$ and $\phi_{\k,R}(0)\ge0$. By comparison principle, we obtain
\begin{equation}\aligned
\k r\le \phi_{\k,R}(r)\le \k r+\f Kr\le\k r+\f{n\k+2}r \qquad \mathrm{for}\ r\in[0,R].
\endaligned
\end{equation}
Combining $\phi_{\k,R}'\ge0$, we have $\k r\le \phi_{\k,R}(r)\le \phi_{\k,R}(1)\le(n+1)\k+2$ for every $r\in[0,1]$.
Let $R\rightarrow\infty$ in $\phi_{\k,R}$, we obtain a function $\phi_\k$ such that $\{(x,\phi_\k(|x|))|\ x\in \R^n\}$ is an \textbf{E}-minimizing self-expanding hypersurface in $\R^{n+1}$. Then its singular set has Hausdorff dimension $\le n-7$ in case $n\ge7$ or is empty in case $n\le6$. As the hypersurface is rotationally symmetric, then its singular set is empty by blowing up argument, which follows that $\phi_\k$ is smooth by the regularity of elliptic equations. Moreover, $\phi_\k$ is a smooth solution to \eqref{Ju} with $\phi_\k'\ge \phi_\k'(0)=0$ and $\lim_{r\rightarrow\infty}\f{\phi_\k(r)}r=\k>0$. In particular,
\begin{equation}\aligned\label{phicomp}
0\le& \phi_\k(r)-\k r\le((n+1)\k+2)\min\{1,r^{-1}\} \qquad \mathrm{for}\ r\in(0,\infty).
\endaligned
\end{equation}
We denote
\begin{equation}\aligned\label{Mk}
M_\k=\{(x,\phi_\k(|x|))|\ x\in\R^n\}
\endaligned
\end{equation}
for such $\phi_\k$, which is a rotational symmetric self-expander.

Now we consider 1-dimensional self-expander, namely, the solution to the following ODE
\begin{equation}\aligned\label{J1u}
\mathcal{J}_1u\triangleq\f{u_{yy}}{1+u_y^2}+\f12yu_y-\f12u=0\qquad \mathrm{on}\ (-\infty,\infty)
\endaligned
\end{equation}
with $\lim_{|y|\rightarrow\infty}\f{u(y)}{|y|}=\k>0$. Assume that $u$ is a solution to \eqref{J1u} on $(-R,R)$ with $u(\pm R)=\k R$.
Then from lemma \ref{comp}, graph$_u$ is a \textbf{E}-minimizing self-expanding curve. Hence $u\le \k R$ on $(-R,R)$, or else the curve $\{(y,\min\{u,\k R\})|\ |y|<\k R\}$ has smaller mass with the weight $e^{\f{|X|^2}4}$ than graph$_u$.
Note that $\mathcal{J}_1(\k y)=0$ and $\mathcal{J}_1(\k R)\le0$.
Analog to the previous argument of the case $n\ge2$, with \emph{a priori} gradient estimate and the continuity method,
there exists a unique solution $\hat{\phi}_{\k,R}$ to \eqref{J1u} with $\hat{\phi}_{\k,R}(\pm R)=\k R$ such that $\hat{\phi}_{\k,R}(y)=\hat{\phi}_{\k,R}(-y)$, $\hat{\phi}_{\k,R}'(0)=0$, and $\hat{\phi}_{\k,R}'\ge0$ on $[0,R)$.

Let
$$w(y)=\k y+\f{\tau}y e^{-\f{1}4y^2}$$
with $\tau>0$ to be defined later. Clearly, $w_y=\k-\f{\tau}2 e^{-\f{1}4y^2}-\f{\tau}{y^2} e^{-\f{1}4y^2}$ and
$$w_{yy}=\left(\f{\tau y}4+\f{\tau}{2y}+\f{2\tau}{y^3}\right) e^{-\f{1}4y^2}.$$
Hence
\begin{equation}\aligned
\mathcal{J}_1w=\left(\f{1}{1+w_y^2}\left(\f{y}4+\f{1}{2y}+\f{2}{y^3}\right)-\f y4-\f1y\right)\tau e^{-\f{1}4y^2}.
\endaligned
\end{equation}
Set $\tau=2e\max\{\k,2\}$. When $y\in(0,2)$, we have $w_y\le\k-\f{\tau}2 e^{-1}-\f{\tau}{y^2} e^{-\f{1}4y^2}\le-\f{\tau}{y^2} e^{-\f{1}4y^2}$, and
\begin{equation}\aligned
\mathcal{J}_1w\le&\left(\left(\f{y}4+\f{1}{2y}+\f{2}{y^3}\right)\f{y^4}{\tau^2}e^{\f12y^2}-\f y4-\f1y\right)\tau e^{-\f{1}4y^2}\\
\le&\left(\left(\f{y^5}4+\f{y^3}{2}+2y\right)\f{1}{\tau^2}e^{\f12y^2}-1\right)\tau e^{-\f{1}4y^2}\\
\le&\left(\f{16}{\tau^2}e^{2}-1\right)\tau e^{-\f{1}4y^2}\le0.
\endaligned
\end{equation}
When $y\ge2$, we have
\begin{equation}\aligned
\mathcal{J}_1w=\left(\f{y}4+\f{1}{2y}+\f{2}{y^3}-\f y4-\f1y\right)\tau e^{-\f{1}4y^2}\le0.
\endaligned
\end{equation}
From comparison principle and $\mathcal{J}_1(\k y)=0$, we obtain
\begin{equation}\aligned
\k y\le \hat{\phi}_{\k,R}(y)\le \k y+\f{\tau}y e^{-\f{1}4y^2}\le\k y+\f{2e(\k+2)}y e^{-\f{1}4y^2} \qquad \mathrm{for}\ y\in[0,R].
\endaligned
\end{equation}
Since $\hat{\phi}'_{\k,R}\ge0$ on $[0,R)$, then $\hat{\phi}_{\k,R}(y)\le\hat{\phi}_{\k,R}(2)\le \k+2$ for $y\in(0,2]$. Therefore,
\begin{equation}\aligned
\k y\le \hat{\phi}_{\k,R}(y)\le\k y+\min\left\{\f{2e(\k+2)}y e^{-\f{1}4y^2},3\k+2\right\} \qquad \mathrm{for}\ y\in[0,R].
\endaligned
\end{equation}
Letting $R\rightarrow\infty$ for $\hat{\phi}_{\k,R}$, we obtain a function $\hat{\phi}_{\k}$ to \eqref{J1u} in $\R$ with $\hat{\phi}_{\k}(y)=\hat{\phi}_{\k}(-y)$ and $\lim_{y\rightarrow\infty}\f{\hat{\phi}_{\k}(y)}y=\k>0$. Further, $\{(y,\hat{\phi}_{\k}(y))|\ y\in \R^1\}$ is an \textbf{E}-minimizing self-expanding curve in $\R^{2}$, which implies that $\hat{\phi}_{\k}$ is smooth. In particular,
\begin{equation}\aligned
0\le& \hat{\phi}_{\k}(y)-\k y\le\min\left\{\f{2e(\k+2)}y e^{-\f{1}4y^2},3\k+2\right\} \qquad \mathrm{for}\ y\in(0,\infty).
\endaligned
\end{equation}
\textbf{Remark:} The above estimate may be not sharp for the decay order.
\begin{lemma}\label{phila0}
Let $\phi_\la$ be a smooth solution to \eqref{Ju} for $n\ge1$ with $\phi_\la'\ge \phi_\la'(0)=0$ and $\lim_{r\rightarrow\infty}\f{\phi_\la(r)}r=\la>0$. Then $\lim_{\la\rightarrow\infty}\phi_\la(0)=\infty$.
\end{lemma}
\begin{proof}
If the Lemma fails, there are a constant $\la_0$ and a sequence $\la_i\rightarrow\infty$ so that for such solutions $\phi_{\la_i}$ there holds $0\le\phi_{\la_i}(0)\le\la_0$ for each $i\ge1$. Denote
$$\Si_{i}=\la_i\big\{(x,s)\in\R^{n+1}|\ s=\phi_{\la_i}(|x|)-\phi_{\la_i}(0),\ x\in\R^n\big\}.$$
Then $\Si_i$ is a rotationally symmetric graph through the origin and
\begin{equation}\aligned\label{Sicontain}
\Si_i\subset&\la_i\{(x,s)\in\R^{n+1}|\ s\ge\max\{\la_i|x|-\phi_{\la_i}(0),0\},\ x\in\R^n\big\}\\
=&\{(x,s)\in\R^{n+1}|\ s\ge\la_i\max\{|x|-\phi_{\la_i}(0),0\},\ x\in\R^n\big\}\\
\subset&\{(x,s)\in\R^{n+1}|\ s\ge\la_i\max\{|x|-\la_0,0\},\ x\in\R^n\big\}.
\endaligned
\end{equation}
Denote $p_i=(0,\phi_{\la_i}(0))\in\R^{n+1}$, and $S_i$ be a portion of $\p B_{k\la_0}(0)$ with $k\ge2$ and boundary $\p S_i=\p\Si_i\cap\p B_{k\la_0}(0)$ such that $S_i$ belongs to the last term in \eqref{Sicontain}.
By the definition of $\phi_\la$, $\Si_i$ is a rigid-body motion of an \textbf{E}-minimizing self-expanding hypersurface, then
\begin{equation}\aligned\label{VolSi***}
\mathcal{H}^n(\Si_i\cap B_{k\la_0}(0))\le\int_{\Si_i\cap B_{k\la_0}(0)}e^{\f{|X+p_i|^2}{4\la_i^2}}d\mu_{\Si_i}\le\int_{S_i}e^{\f{|X+p_i|^2}{4\la_i^2}}d\mu_{S_i}=\mathcal{H}^n(S_i)e^{\f{(k+1)^2\la_0^2}{4\la_i^2}}.
\endaligned
\end{equation}
From the equation of self-expanders, the mean curvature of $\Si_i$ is bounded by $\f{(k+1)\la_0}{2\la_i^2}$ in $B_{k\la_0}(0)$.
Combining Theorem 17.6 of \cite{S} and $0\in\Si_i$, for the sufficiently large $i>0$ we have $\mathcal{H}^n(\Si_i\cap B_{k\la_0}(0))\ge\f{3}4\omega_nk^n\la_0^n$,
where $\omega_n$ is the volume of $n$-dimensional unit ball in $\R^n$. Combining \eqref{VolSi***}, we have
$$\mathcal{H}^n(S_i)\ge\f{1}2\omega_nk^n\la_0^n$$
for the sufficiently large $i>0$. From \eqref{Sicontain}, for the sufficiently large $i>0$ we have
$$\limsup_{i\rightarrow\infty}\mathcal{H}^n(S_i)\le\int_{B_{\la_0}(0)}\f{k\la_0}{\sqrt{k^2\la_0^2-|x|^2}}dx\le\f{k}{\sqrt{k^2-1}}\omega_n\la_0^n.$$
It's a contradiction for the large $k>0$.
This completes the proof.
\end{proof}

Now we give a Bernstein type theorem for self-expanders. One may compare it with Lemma \ref{area-minimSE}.
\begin{theorem}
Any smooth self-expander whose tangent cone at infinity is a hyperplane with an integer multiplicity, must be a hyperplane.
\end{theorem}
\begin{proof}
Let $\Si$ be a smooth self-expander $\R^{n+1}$ with tangent cone $\R^n$ at infinity.
For $n\ge3$, we have two barrier functions $\pm\ep r\pm\f{(n-1)\ep}r$ from \eqref{Jkn-1k}. In other words, by maximum principle $\Si$ is pinched by two graphs
$$\left\{\left(x,\ep|x|+\f{(n-1)\ep}{|x|}\right)\bigg|\ x\in\R^n\setminus\{0\}\right\},\quad \mathrm{and}\ \ \left\{\left(x,-\ep|x|-\f{(n-1)\ep}{|x|}\right)\bigg|\ x\in\R^n\setminus\{0\}\right\}.$$
Letting $\ep\rightarrow0$ implies that $\Si$ is $\R^n$. If $\Si$ is a self-expander in $\R^k$, then $\Si\times\R^2$ is also a self-expander in $\R^{k+2}$. Therefore, we complete the proof.
\end{proof}

\section{Asymptotic analysis for self-expanders}

\subsection{Allard's type regularity theorem and decay at infinity for varifold self-expanders}
Let $T$ be an $n$-dimensional varifold self-expander in $\R^{n+1}$ with boundary $\p T$ in $B_{R_0}$ for some $R_0>0$.
Set
$$\mathcal{T}: t\in(0,\infty)\rightarrow \mathcal{T}_t=\sqrt{t}T.$$
We recall Huisken's monotonicity formula for the Brakke flow $\mathcal{T}$ with boundary as follows (see formula (1.2) in \cite{CM2} for instance with the test function replacing by $\chi_{_{B_{R}}}(t^{-\f12}\cdot)$).
Here $\chi_{_{B_{R}}}$ denotes the characteristic function on $B_R$, i.e., it is equal to 1 on $B_{R}$ and 0 on $\R^{n+1}\setminus B_R$.
Let $\r(X,t)=\left(4\pi(t_0-t)\right)^{-n/2}\exp\left(-\f{|X-X_0|^2}{4(t_0-t)}\right)$ with $0<t<t_0$, then for almost all $R\ge R_0$
\begin{equation}\aligned\label{MonotoB}
\f{d}{d t}\int_{\mathcal{T}_{t}\setminus B_{\sqrt{t}R}}\r d\mu_{\mathcal{T}_t}=&\f12 t^{-\f12}\int_{\p(\mathcal{T}_{t}\cap B_{\sqrt{t}R})}\left(\f{|X^T|}{t}+\left\lan\f{(X-X_0)^T}{t_0-t},\f{X^T}{|X^T|}\right\ran\right)\r d\mu_{\mathcal{T}_t}\\
&-\f14\int_{\mathcal{T}_{t}\setminus B_{\sqrt{t}R}}\left|\f{X^N}{t}+\f{(X-X_0)^N}{t_0-t}\right|^2\r d\mu_{\mathcal{T}_t}.
\endaligned
\end{equation}
Here, $d\mu_{\mathcal{T}_t}$ is the volume element of $\mathcal{T}_t$.
We assume that the multiplicity function of the varifold $T$ is 1 a.e. and $\f1rT$ converges to $|\llbracket C\rrbracket|$ locally in the varifold sense as $r\rightarrow\infty$. Here, $C$ is a $C^{3,\a}$-regular cone in $\R^{n+1}$ for some $\a\in(0,1)$, i.e., a cone over a compact, embedded $C^{3,\a}$-hypersurface without boundary in the sphere $\S^n$.

For any $\textbf{X}=(X,t)\in\R^{n+1}\times\R$ and $0<r<\sqrt{t}$, we define the \emph{Gaussian density ratio} of $\mathcal{T}$ at $\textbf{X}$ with radius $r$ by
\begin{equation}\aligned
\Th(\mathcal{T},\textbf{X},r)=\int_{Y\in \mathcal{T}_{t-r^2}}\f1{(4\pi r^2)^{n/2}}e^{-\f{|Y-X|^2}{4r^2}}d\mu_{\mathcal{T}_{t-r^2}}=\int_{Z\in\sqrt{\f{t}{r^2}-1}T-\f{X}r}\f1{(4\pi )^{n/2}}e^{-\f{|Z|^2}{4}},
\endaligned
\end{equation}
and $\Th(\mathcal{T},\textbf{X},\sqrt{t})=\lim_{r\rightarrow\sqrt{t}}\Th(\mathcal{T},\textbf{X},r)$.
Due to $C=\lim_{r\rightarrow\infty}\f1rT$ in the varifold sense, it follows that for any fixed $t>0$
\begin{equation}\aligned
\sup_{|X|\ge\f1{\ep}}\Th(\mathcal{T},\textbf{X},\sqrt{t})=&\sup_{|X|\ge\f1\ep}\int_{Z\in C-\f{X}{\sqrt{t}}}\f1{(4\pi )^{n/2}}e^{-\f{|Z|^2}{4}}\\
=&\sup_{|X|\ge\f1{\ep\sqrt{t}}}\int_{Z\in C-X}\f1{(4\pi )^{n/2}}e^{-\f{|Z|^2}{4}}.\\
\endaligned
\end{equation}
Thus for any $0<\ep<1$, there is a constant $\ep'=\ep'(\ep)>0$ such that
$$\sup_{|X|\ge\f1{\ep'},t\in(0,2]}\Th(\mathcal{T},\textbf{X},\sqrt{t})\le\sup_{|X|\ge\f1{\sqrt{2}\ep'}}\int_{Z\in C-X}\f1{(4\pi )^{n/2}}e^{-\f{|Z|^2}{4}}<1+\f{\ep}2.$$
From \eqref{MonotoB} with almost each $R\in(R_0,R_0+1)$, for any $0<r\le\sqrt{t}\le\sqrt{2}$, any $\textbf{X}=(X,t)$ with $|X|\ge\f1{\ep'}$, and the sufficiently small $\ep'>0$ we get
\begin{equation}\aligned\label{ThTXrest*}
&\Th(\mathcal{T},\textbf{X},r)\le\int_{Y\in \mathcal{T}_{t-r^2}\setminus B_{\sqrt{t-r^2}R}}\f1{(4\pi r^2)^{n/2}}e^{-\f{|Y-X|^2}{4r^2}}d\mu_{\mathcal{T}_{t-r^2}}+\f{\ep}4\\
<&\int_0^{t-r^2}\f12s^{-\f12}\left(\int_{Y\in \p\left(\mathcal{T}_{s}\setminus B_{\sqrt{s}R}\right)}\left(\f{|Y|}{s}+\f{|Y-X|}{t-s}\right)\f{e^{-\f{|Y-X|^2}{4(t-s)}}}{(4\pi(t-s))^{n/2}}\right)ds\\
&+\int_{Z\in C}\f1{(4\pi r^2)^{n/2}}e^{-\f{|Z-X|^2}{4r^2}}+\f{\ep}4\\
<&\f{1}{(4\pi r^2)^{n/2}}e^{-\f{\left|\f1{\ep'}-\sqrt{2}R\right|^2}{4r^2}}\int_0^{t-r^2}\left(\int_{\p\left(\mathcal{T}_{s}\setminus B_{\sqrt{s}R}\right)}\left(\f{R}{2s}+\f{\left|\sqrt{2}R+\f1{\ep'}\right|}{2r^2\sqrt{s}}\right)\right)ds\\
&+\int_{Z\in C-\f{X}{r}}\f1{(4\pi )^{n/2}}e^{-\f{|Z|^2}{4}}+\f{\ep}4.
\endaligned
\end{equation}
Note that $\mathcal{T}_s$ converges as $s\rightarrow0$ to $C$ in the varifold sense, then combining Fubini theorem we have
\begin{equation}\aligned
\f{1}{(4\pi r^2)^{n/2}}e^{-\f{\left|\f1{\ep'}-\sqrt{2}R\right|^2}{4r^2}}\int_0^{t-r^2}\left(\int_{\p\left(\mathcal{T}_{s}\setminus B_{\sqrt{s}R}\right)}\left(\f{R}{2s}+\f{\left|\sqrt{2}R+\f1{\ep'}\right|}{2r^2\sqrt{s}}\right)\right)ds<\f{\ep}4
\endaligned
\end{equation}
for some $R\in(R_0,R_0+1)$ and each sufficiently small $\ep'>0$. Combining \eqref{ThTXrest*} one has
\begin{equation}\aligned\label{ThTXrest**}
&\Th(\mathcal{T},\textbf{X},r)\le\int_{Z\in C-\f{X}{r}}\f1{(4\pi )^{n/2}}e^{-\f{|Z|^2}{4}}+\f{\ep}2.
\endaligned
\end{equation}
for any $0<r\le\sqrt{t}\le\sqrt{2}$, and any $\textbf{X}=(X,t)$ with $|X|\ge\f1{\ep'}$. In particular, taking $t=1+r^2$ and $r\rightarrow0$ implies
\begin{equation}\aligned\label{monoTXr}
\limsup_{r\rightarrow0}\left(\int_{Y\in T}\f1{(4\pi r^2)^{n/2}}e^{-\f{|Y-X|^2}{4r^2}}d\mu_T\right)
\le1+\ep
\endaligned
\end{equation}
for any $|X|\ge\f1{\ep'}$.

\begin{theorem}\label{weaksmooth}
If $T$ is a multiplicity one $n$-varifold self-expander in $\R^{n+1}$ with boundary $\p T$ in $B_{R_0}$, and its tangent cone at infinity is $|\llbracket C\rrbracket|$. Then there exists a small constant $\ep_1>0$ depending only on $n,C$ and $\sup_{x\in\p T}|x|$ such that $\mathrm{spt}T\setminus B_{\f1{\ep_1}}$ is smooth.
\end{theorem}
\begin{proof}
Let $\mathcal{H}^{k}(K)$ denote the $k$-dimensional Hausdorff measure of $K$ for any constant $k\ge0$ and any set $K$ in Euclidean space.
We claim that for every $X\in T\llcorner\left(\R^{n+1}\setminus B_{\f{2}{\ep'}}\right)$ (one may find the definition of the natation '$\llcorner$' in Chapter 15 in \cite{S}), there is a sufficiently small $\r_0>0$ such that
\begin{equation}\aligned\label{dens}
\f{\mathcal{H}^n(T\llcorner B_{\r_0}(X))}{\omega_n\r_0^n}\le1+2\ep,
\endaligned
\end{equation}
where $\omega_n$ is the volume of $n$-dimensional unit ball $B_1\subset\R^n$. If not, there exist a point $X_0\in T\llcorner\left(\R^{n+1}\setminus B_{\f{2}{\ep'}}\right)$ and a sequence $\r_1>\r_2>\cdots>\r_i\rightarrow0$ such that
$$\f{\mathcal{H}^n(T\llcorner B_{\r_i}(X_0))}{\omega_n\r_i^n}>1+2\ep,$$
for all $i\ge1$. By the monotonicity identity 17.3 in \cite{S} and the definition of generalized mean curvature of varifold self-expander $T$, we get
$$\f{d}{d\r}\left(\f{\mathcal{H}^n(T\llcorner B_{\r}(X_0))}{\omega_n\r^n}\right)\ge-\f12\left(|X_0|+\r\right).$$
Suppose that $\r_1$ is sufficiently small, then
\begin{equation}\aligned\label{Trep}
\f{\mathcal{H}^n(T\llcorner B_{\r}(X_0))}{\omega_n\r^n}>1+\f32\ep,
\endaligned
\end{equation}
for every $\r\in(\r_i,\r_1]$. Clearly, we can force $\r_i\rightarrow0$, such that the inequality \eqref{Trep} holds on $(0,\r_1]$.

For $\textbf{X}_0=(X_0,1+r^2)$, the Gaussian density ratio
\begin{equation}\aligned
\Th(\mathcal{T},\textbf{X}_0,r)=&\int_{Y\in T}\f1{(4\pi r^2)^{n/2}}e^{-\f{|Y-X_0|^2}{4r^2}}d\mu_T\ge\int_{Y\in T\cap B_{\r_1}(X_0)}\f1{(4\pi r^2)^{n/2}}e^{-\f{|Y-X_0|^2}{4r^2}}d\mu_T\\
=&\int_0^{\r_1}\f1{(4\pi r^2)^{n/2}}e^{-\f{\r^2}{4r^2}}\f{d}{d\r}\mathcal{H}^n(T\llcorner B_\r(X_0))d\r\\
=&\int_0^{\r_1}\f1{(4\pi r^2)^{n/2}}e^{-\f{\r^2}{4r^2}}\f{\r}{2r^2}\mathcal{H}^n(T\llcorner B_\r(X_0))d\r\\
\ge&\left(1+\f32\ep\right)\omega_n\int_0^{\r_1}\f1{(4\pi r^2)^{n/2}}e^{-\f{\r^2}{4r^2}}\f{\r}{2r^2}\r^nd\r\\
=&\left(1+\f32\ep\right)\f{\omega_n}{\pi^{n/2}}\int_0^{\f{\r_1^2}{4r^2}}t^{n/2}e^{-t}dt,
\endaligned
\end{equation}
where $\omega_n=\f{\pi^{n/2}}{\G(\f n2+1)}$. 
Then for any sufficiently small $r\in(0,\r_1)$ we have
$$\Th(\mathcal{T},\textbf{X}_0,r)>1+\ep,$$
which is in contradiction to \eqref{monoTXr}. Hence the inequality \eqref{dens} holds. By Allard's regularity theorem(see Theorem 24.2 in \cite{S} for instance) and the equation of self-expanders, we complete the proof.
\end{proof}
Remark. The area ratios in Theorem \ref{weaksmooth} can be obtained directly from estimates on the initial
Gaussian density ratios for small $r$ with the initial regular cone by \cite{K} or \cite{W05}, as the reviewer pointed out.

In the above theorem, we have obtained that
$t\in(0,\infty)\rightarrow \sqrt{t}\left(\mathrm{spt}T\right)$ is a smooth mean curvature flow on $\left(\R^{n+1}\setminus B_1\right)\times(0,\ep_1^2]$.
Furthermore, we have a more stronger convergence as follows.
\begin{lemma}
As $t\rightarrow0$, $\sqrt{t}T$ converges to the regular cone $C$ in $C^{3}$-norm.
\end{lemma}
\begin{proof}
From \eqref{ThTXrest**}, $\Th(\mathcal{T},\textbf{X},r)<1+\ep$ for any $0<r\le\sqrt{t}\le\sqrt{2}$, and any $\textbf{X}=(X,t)$ with $|X|\ge\f1{\ep'}$.
By White's regularity theorem \cite{W05}, we have
\begin{equation}\aligned\label{RfConA}
\lim_{|X|\rightarrow\infty}|A_T(X)|=0,
\endaligned
\end{equation}
where $A_{T}(X)$ is the second fundamental form of $\mathrm{spt}T$ at $X$.

If there is a sequence of points $\textbf{X}_i=(X_i,t_i)\in \sqrt{t_i}\left(\mathrm{spt}T\right)$ with $1\le |X_i|\le2$, $t_i\rightarrow0$ and $\lim_{i\rightarrow\infty}d_H\left(C,X_i/\sqrt{t_i}\right)>0$.
Here, $d_H$ denotes the Hausdorff distance.
Then for any $\ep>0$, for the sufficiently small $\ep>0$, we have
\begin{equation}\aligned
\int_{Z\in C-\f{X_i}{\sqrt{t_i}}}\f1{(4\pi )^{n/2}}e^{-\f{|Z|^2}{4}}<1-\ep.
\endaligned
\end{equation}
Combining \eqref{ThTXrest**}, we have
\begin{equation}\aligned\label{ThTXrest***}
1=\Th(\mathcal{T},\textbf{X}_i,r_i)\le\int_{Z\in C-\f{X_i}{\sqrt{t_i}}}\f1{(4\pi )^{n/2}}e^{-\f{|Z|^2}{4}}+\f{\ep}2\le1-\ep+\f{\ep}2=1-\f{\ep}2.
\endaligned
\end{equation}
It contradicts to $0<\ep<1$. Hence, one has
\begin{equation}\aligned
\lim_{r\rightarrow\infty}d_H\left(\left(\mathrm{spt}T\right)\cap B_{2r}\setminus B_r,C\cap B_{2r}\setminus B_r\right)=0.
\endaligned
\end{equation}
Let $\nu_{T}$ denote the unit normal vector of $T$ and $\nu_{C}$ denote the unit normal vector of $C$. Now we claim that
\begin{equation}\aligned\label{ConnuCnuT}
\lim_{|X|\rightarrow\infty}\sup_{\xi,\e\in B_1(X),\, \xi\in C,\, \e\in T}\left|\nu_C(\xi)-\nu_T(\e)\right|=0.
\endaligned
\end{equation}
Namely, for any compact set $K$, $\sqrt{t}\left(\mathrm{spt}T\right)\cap K$ converges to $C\cap K$ in $C^1$-norm.
Note that the cone $C$ is of $C^{3,\a}$ class. Analog to the step 1 in the proof of Theorem 1.1 in \cite{Wm} (with Schauder fixed-point theorem and Schauder estimates for linear parabolic equations), we get $\sqrt{t}T$ converging to the regular cone $C$ in $C^{3}$-norm.

Now let us show the claim \eqref{ConnuCnuT}. If \eqref{ConnuCnuT} fails, then there exist a sequence of points $X_i$ with $|X_i|\rightarrow\infty$ and a constant $\de_*>0$ such that
\begin{equation}\aligned\label{iConnuCnuT}
\lim_{i\rightarrow\infty}\sup_{\xi,\e\in B_1(X_i),\, \xi\in C,\, \e\in T}\left|\nu_C(\xi)-\nu_T(\e)\right|=\de_*.
\endaligned
\end{equation}
For any fixed $X\in\mathrm{spt} T$, we define a cylindrical domain $\mathbf{C}_{r}(X)$ by $\{y+s\nu_T(X)|\ y\bot\nu_T,\, |y-X|<10r/\de_*,\, |s|<\de_*r\}$ for all $r>0$ and $l\ge1$. Now we define a sequence Lipschitz function $\phi_i$ on $\R^{n+1}$ with Lipschitz constant $\le2/\de_*$ such that $\phi_i=1$ in $\mathbf{C}_{1}(X_i)$, $\phi_i=0$ in $\R^{n+1}\setminus\mathbf{C}_{2}(X_i)$. From the mean curvature flow of $\sqrt{t}\left(\mathrm{spt}T\right)$, we have
\begin{equation}\aligned
\int_{\sqrt{t}T}\phi_i-\int_{\sqrt{s}T}\phi_i\le&-\int_s^t\int_{\sqrt{\tau}T}|H_{\sqrt{\tau}T}|^2\phi_i d\mu_{\mathcal{T}_\tau}d\tau
+\int_s^t\int_{\sqrt{\tau}T}|H_{\sqrt{\tau}T}|\cdot|\bn\phi_i|d\mu_{\mathcal{T}_\tau}d\tau\\
\le&\f2{\de_*}\int_s^t\int_{\sqrt{\tau}T\cap \mathbf{C}_2(X_i)}|H_{\sqrt{\tau}T}|d\mu_{\mathcal{T}_\tau}d\tau
\endaligned
\end{equation}
for each $0<s<t$, where $H_{\sqrt{t}T}$ is the mean curvature of $\sqrt{t}T$.
From \eqref{RfConA}, there is a sequence $\tau_i\rightarrow0$ such that
$|H_{\sqrt{\tau}T}|\le \f{\tau_i}{\sqrt{\tau}}$ on $\sqrt{\tau}T\cap \mathbf{C}_2(X_i)$. With \eqref{ThTXrest**}, there is a constant $c$ depending only on $n$ and the cone $C$ such that
$\mathcal{H}^n\left(\sqrt{\tau}T\cap \mathbf{C}_2(X_i)\right)\le c\de_*^{-n}$. Then we get
\begin{equation}\aligned
\int_{\sqrt{t}T}\phi_i-\int_{\sqrt{s}T}\phi_i\le&\f2{\de_*}\int_s^t c\de_*^{-n}\f{\tau_i}{\sqrt{\tau}}d\tau\le 4c\tau_i\de_*^{-n-1}\sqrt{t},
\endaligned
\end{equation}
and letting $s\rightarrow0$ implies
\begin{equation}\aligned\label{TCtaui}
\int_{T}\phi_i\le\int_{C}\phi_i+4c\tau_i\de_*^{-n-1}.
\endaligned
\end{equation}
By the definition of $\phi_i$, we have $\int_{T}\phi_i\ge\omega_n(\f{10}{\de_*})^n$ for the sufficiently large $i\ge0$.
With \eqref{iConnuCnuT}, $\int_{C}\phi_i\le\f12\omega_n(\f{10}{\de_*})^n$ for the sufficiently large $i\ge0$. It's a contradiction to \eqref{TCtaui}.
Hence we complete the proof of \eqref{ConnuCnuT}, and then complete the proof of this lemma.
\end{proof}
From the above lemma, for the sufficiently small $\ep_1>0$,
there is a constant $c_{\ep_1}$ depending only on $n,\ep_1$ and the cone $C$ such that for any $\mathbf{X}=(X,t)\in\left(\overline{B_2}\setminus B_1\right)\times(0,\ep_1^2]$
$$|A_{\sqrt{t}T}(X)|\le c_{\ep_1},$$
where $A_{\sqrt{t}T}(X)$ is the second fundamental form of $\sqrt{t}\left(\mathrm{spt}T\right)$ at $X$. Namely, for any $r\ge\f1{\ep_1}$
\begin{equation}\aligned\label{nakA}
|A_T|\le \f{c_{\ep_1}}{r}
\endaligned
\end{equation}
on $\left(B_{2r}\setminus B_r\right)\cap \mathrm{spt}T$.
Set $A_C$ be the second fundamental form of $C\setminus\{0\}$. Then there is a constant $c>0$ such that the norm of the second fundamental form $|A_C(x)|\le\f{c}{|x|}$ at $x\in C\setminus\{0\}$. Let us recall that $c$ denotes a positive constant depending only on $n$ and the cone $C$ but will be allowed to change from line to line.

For any point $z\in \mathrm{spt}T$ we define a subset $E_z$ of $C$ by
$$\{\xi_z\in C\big|\, |z-\xi_z|=\mathrm{dist}(z,C)\triangleq\inf_{y\in C} |y-z|\}.$$
From $\f1r\left(\mathrm{spt}T\right)\rightarrow C$ locally in $C^3$-sense, $|\xi_z|\ge\f{|z|}2$ for the large $|z|>0$. Then we have
$$|z-\xi_z|\cdot|A_C(\xi_z)|\le c\f{\mathrm{dist}(z,C)}{|\xi_z|}\le 2c\f{\mathrm{dist}(z,C)}{|z|}$$
and
$$ \limsup_{z\in \mathrm{spt}T,|z|\rightarrow\infty}\left(\sup_{\xi_z\in E_z}|z-\xi_z|\cdot|A_C(\xi_z)|\right)\le 2c\limsup_{z\in \mathrm{spt}T,|z|\rightarrow\infty}\left(\f{\mathrm{dist}(z,C)}{|z|}\right)=0.$$
So there is a constant $R_1>\f1{\ep_1}$ such that for each $z\in \mathrm{spt}T\setminus B_{R_1}$, $E_z$ includes only one point denoted by $z_{_C}$(see also the proof of Lemma 2.3 in \cite{W}). Hence we can define a function $f$ by
\begin{equation}\aligned\label{fXy}
z_{_C}+f(z_{_C})\nu_C=z,
\endaligned
\end{equation}
where $\nu_C$ is the unit normal vector of $C$ at $z_{_C}$ pointing to $z$. Since $\f1r\left(\mathrm{spt}T\right)$ converges to the cone $C$ in $C^3$-sense on $\overline{B_4}\setminus B_{\f12}$ as $r\rightarrow\infty$,
\begin{equation}\aligned\label{na0123finfty}
\limsup_{y\in C,|y|\rightarrow\infty}\left(\sum_{i=0}^3|y|^{i-1}|\na^i_Cf(y)|\right)=0.
\endaligned
\end{equation}
So there is a constant $R_2>R_1$ such that $f$ satisfies
\begin{equation}\aligned\label{na01fep2}
\sum_{i=0}^3|y|^{i-1}|\na^i_Cf(y)|\le \f1{100}
\endaligned
\end{equation}
for every $y\in C\setminus B_{R_2}$. Furthermore, we have the following estimates.
\begin{lemma}\label{morepresice01order}
There is a constant $R_3\ge R_2$ depending on $n,C$ and $\sup_{x\in\p T}|x|$, such that at $y\in C\setminus B_{R_3}$,
\begin{equation}\aligned\label{naCf518}
|\na_C^{i}f(y)|\le c|y|^{-i-1}\qquad \mathrm{for}\ i=0,1.
\endaligned
\end{equation}
\end{lemma}
\begin{proof}
For any $x_0\in C\setminus B_{R_2}$, there is a unique hyperplane $P\subset\R^{n+1}$ through $0,x_0$ such that $P$ is tangent to $C$ at $x_0$. In a small tubular neighborhood of the radial line $\{tx_0|\ t\ge R_3\}$ for a sufficiently large constant $R_3\ge R_2$, $\mathrm{spt}T$ can be seen as a graph over an open set $P_0\subset P$ with the graphic function $u$ and $tx_0\in P_0$ for all $t\ge R_3$. From \eqref{na01fep2}, one has
\begin{equation}\aligned\label{na01uep2}
|u(x)|+|x|\cdot|Du(x)|\le\f1{50}|x|
\endaligned
\end{equation}
and
\begin{equation}\aligned\label{na123uep2}
|D^{i+1}u(x)|\le c|x|^{-i}
\endaligned
\end{equation}
for every $x\in P_0$, and $1\le i\le2$. Clearly, $u(x_0)=f(x_0)$. The normal vector to $T$ at $x_0$ is
$$\f{-\na_Cf(x_0)+\nu_C(x_0)}{\sqrt{1+|\na_Cf(x_0)|^2}}=\f{-Du(x_0)+\nu_C(x_0)}{\sqrt{1+|Du(x_0)|^2}},$$
which implies $\na_Cf(x_0)=Du(x_0)$. Hence from \eqref{na0123finfty} we obtain
\begin{equation}\aligned\label{tutx01}
\lim_{t\rightarrow\infty}\big(t^{-1}|u(tx_0)|+|x_0|\cdot|Du(tx_0)|\big)=0,
\endaligned
\end{equation}
and then
\begin{equation}\aligned
\f{u(x_0)}{|x_0|}=-\int_1^\infty\f{\p}{\p t}\left(\f{u(tx_0)}{t|x_0|}\right)dt=\int_1^\infty\f{-tx_0\cdot Du(tx_0)+u(tx_0)}{t^2|x_0|}dt.
\endaligned
\end{equation}
Combining \eqref{Graphu}, \eqref{na01uep2} and \eqref{na123uep2} one gets
\begin{equation}\aligned
|u(x_0)|\le&\int_1^\infty\f{\left|-tx_0\cdot Du(tx_0)+u(tx_0)\right|}{t^2}dt=\f12\int_1^\infty t^{-2}\left|\sum_{i,j}g^{ij}u_{ij}\Big|_{tx_0}\right|dt\\
\le&\f{\sqrt{n}}2\int_1^\infty t^{-2}\left|D^2u(tx_0)\right|dt\le\f{\sqrt{n}}2\int_1^\infty t^{-2}\f{c}{t|x_0|}dt=\f{\sqrt{n}c}{4|x_0|}.
\endaligned
\end{equation}
Taking derivative on $x_k$ for \eqref{Graphu}, we obtain
\begin{equation}\aligned
2\p_{x_k}\left(\sum_{i,j}g^{ij}u_{ij}\right)=-\sum_ix_iu_{ik}.
\endaligned
\end{equation}
Then by \eqref{na01uep2}\eqref{na123uep2}, we have
\begin{equation}\aligned
\left|\sum_ix_iu_{ik}\right|\le c|x|^{-2}.
\endaligned
\end{equation}
Hence
\begin{equation}\aligned
\left|u_k(x_0)\right|=&\left|\int_1^\infty\f{\p}{\p t}u_k(tx_0)dt\right|=\left|\int_1^\infty\sum_ix_0\cdot Du_{k}(tx_0)dt\right|\\
\le&\int_1^\infty\f{c}{t^3|x_0|^2}dt=\f{c}{2|x_0|^2}.
\endaligned
\end{equation}
We complete the proof.
\end{proof}

Let $S$ be a scaling $n$-dimensional smooth self-expander outside a compact set, whose mean curvature
\begin{equation}\aligned\label{Sp}
H_S=p\lan X,\nu\ran
\endaligned
\end{equation}
for some constant $p\in(0,1]$. Suppose that the tangent cone of $\llbracket S\rrbracket$ at infinity is also $\llbracket C\rrbracket$. Similarly, we can define a function $\tilde{f}$ on $S\setminus B_{\f12r_{_S}}$ such that $\mathrm{spt}T$ is a graph over $S\setminus B_{\f12r_{_S}}$ with the graphic function $\tilde{f}$ outside a compact set. Here, $r_{_S}\ge4R_3$ is a positive constant depending only on $n,C$ and $\sup_{x\in\p S}|x|$. For any $z\in \mathrm{spt}T\setminus B_{\f23r_{_S}}$ there is a unique point $z_{_S}\in S$ such that $\tilde{f}(z_{_S})=|z-z_{_S}|=\mathrm{dist}(z,S)\triangleq\inf_{y\in S}|z-y|$. Then $z= z_{_S}+\tilde{f}(z_{_S})\nu_S$ with the normal vector $\nu_S$ of $S$ at $z_{_S}$. Further,
\begin{equation}\aligned\label{YStf0}
\limsup_{z_{_S}\in S,|z_{_S}|\rightarrow\infty}\left(\sum_{i=0}^2|z_{_S}|^{i-1}\left|\na^i_S\tilde{f}(z_{_S})\right|\right)=0,
\endaligned
\end{equation}
where $\na_S$ is the Levi-Civita connection of $S$.

Assume that $S$ is a graph over $C\setminus B_{\f14r_{_S}}$ with the graphic function $\zeta$ outside a compact set.
There is a point $z_{_C}\in C$ such that $z_{_C}+f(z_{_C})\nu_C=z$. By the definition of $\tilde{f}$ we have
\begin{equation}\aligned\label{YStf1}
\tilde{f}(z_{_S})\le |z-z_{_C}-\zeta(z_{_C})\nu_C|\le|z-z_{_C}|+|\zeta(z_{_C})|\le c|z_{_S}|^{-1}.
\endaligned
\end{equation}
There exists a point $\hat{z}_{_C}\in C$ such that $\hat{z}_{_C}+\z(\hat{z}_{_C})\nu_C=z_S$. The slopes of $f$ at $z_{_C}$ and $\z$ at $\hat{z}_{_C}$ are both bounded by $c|z_{_C}|^{-2}$. Let $\th_{z}$ denote the angle between the tangent space of $S$ at $z_{_S}$, and the tangent space of $T$ at $z$, then $|\th_z|\le c|z_{_S}|^{-2}$. Hence the slope of $\tilde{f}$ at $z_{_S}$ is bounded by $c|z_{_S}|^{-2}$, namely,
\begin{equation}\aligned\label{YStf2}
|\na_S\tilde{f}(z_{_S})|\le c|z_{_S}|^{-2}\qquad \mathrm{on}\ \ S\setminus B_{r_{_S}}.
\endaligned
\end{equation}

\subsection{Linearisation and the Jacobi field operator on self-expanders}
For any smooth hypersurface $S\subset\R^{n+1}$ (maybe with boundary) and any $C^2$-function $\mathcal{F}$ on $S$, we define
\begin{equation}\aligned\label{LSFSE}
L_S\mathcal{F}\triangleq\De_S\mathcal{F}+\f12\lan x,\na_S\mathcal{F}\ran+\left(|A_S|^2-\f12\right)\mathcal{F}
\endaligned
\end{equation}
at $x\in S$, where $\De_S$, $\na_S$, $A_S$ and $H_S$ are Laplacian, the Levi-Civita connection, the second fundamental form and mean curvature of $S$ in $\R^{n+1}$, respectively.

Let $\{e_i\}_{i=1}^n$ be an orthonormal basis at any considered point of $C\setminus\{0\}$, we define the mean curvature of $C$ by $H_C=\lan\overline{\na}_{e_i}e_i,\nu_C\ran$.
Analog to self-shrinkers (Lemma 2.4 of L. Wang \cite{W}), we obtain the following lemma.
\begin{lemma}\label{5.4Asp}
There are constants $R_4,c>0$ depending only on $n,C$ and $\sup_{x\in\p T}|x|$, such that at $x\in C\setminus B_{R_4}$,
\begin{equation}\aligned\label{LCDeCACQf}
L_Cf+H_C=Q_{C}(x,f,\na_Cf,\na^2_Cf),
\endaligned
\end{equation}
where $L_Cf$ is defined by \eqref{LSFSE} and the function $Q_{C}(x,f,\na_Cf,\na^2_Cf)$ satisfies
\begin{equation}\aligned
|Q_{C}(x,f,\na_Cf,\na^2_Cf)|\le c|x|^{-3}\left(|x|\cdot|f|+|\na_Cf|\right).
\endaligned
\end{equation}
\end{lemma}
\begin{proof}
In a neighborhood of $x_0\in C$ with $|x_0|> R_4$, there is a domain $U_0$ including 0 in $\R^n$ so that we can choose a local parametrization of $C$, and a map $F:\ U_0\rightarrow C$ such that $F(0)=x_0$, $\lan\p_iF(0),\p_jF(0)\ran=\de_{ij}$ and $\p_{ij}^2F(0)=h_{ij}\mathbf{n}(x_0)$ with $h_{ij}=\lan\overline{\na}_{\p_iF}\p_jF,\mathbf{n}\ran$ and $h_{ij}(0)=0$ if $i\neq j$. Here $\mathbf{n}(x_0)=\nu_C(x_0)$. Then we have $\p_i\mathbf{n}=-\sum_jh_{ij}\p_jF$. In a neighborhood of $y_0=x_0+f(x_0)\mathbf{n}(x_0)$, there is a local parameterization of $M$, $\widetilde{F}:\ U_0\rightarrow M$ such that
$$\widetilde{F}(p)=F(p)+f(p)\mathbf{n}(p),$$
for $p\in U_0$. Here, one has identified $f(p)$ and $\mathbf{n}(p)$ with $f(F(p))$ and $\mathbf{n}(F(p))$ as in \cite{W}, respectively. For avoiding confusion, we use the notation $\mathbf{n}$ instead of $\nu_C$ in this proof. At $p=0\in U_0$, we have
\begin{equation}\aligned
\p_i\widetilde{F}=\p_iF+(\p_if)\mathbf{n}+f\p_i\mathbf{n}=\p_iF+(\p_if)\mathbf{n}-f\sum_jh_{ij}\p_jF=(1-h_{ii}f)\p_iF+(\p_if)\mathbf{n},
\endaligned
\end{equation}
and the normal vector to $M$ at $\widetilde{F}(0)$
\begin{equation}\aligned
\mathbf{N}=-\sum_k\left(\prod_{l\neq k}(1-h_{ll}f)\right)(\p_kf)\p_kF+\prod_k(1-h_{kk}f)\mathbf{n}.
\endaligned
\end{equation}
By the proof of Lemma 2.4 in \cite{W}, we get at $p=0$
\begin{equation}\aligned
\left\lan \widetilde{F},\mathbf{N}\right\ran=-\sum_k\left(\prod_{l\neq k}(1-h_{ll}f)\right)\lan F,\p_kF\ran(\p_kf)+\left(\lan F,\mathbf{n}\ran+f\right)\prod_k(1-h_{kk}f),
\endaligned
\end{equation}
and
\begin{equation}\aligned
\Big\lan &\p_{ij}^2\widetilde{F},\mathbf{N}\Big\ran=h_{ii}(\p_if)(\p_jf)\prod_{k\neq i}(1-h_{kk}f)+h_{jj}(\p_if)(\p_jf)\prod_{k\neq j}(1-h_{kk}f)\\
&+\left(h_{ij}-h_{ii}h_{jj}\de_{ij}f+\p_{ij}^2f\right)\prod_k(1-h_{kk}f)+f\sum_k\left(\prod_{l\neq k}(1-h_{ll}f)\right)(\p_jh_{ik})(\p_kf).
\endaligned
\end{equation}
Note that $\lan F(0),\mathbf{n}\ran=\lan x_0,\mathbf{n}\ran=0$ as $C$ is a cone.
By \eqref{naCf518}, clearly we have
\begin{equation}\aligned\label{hFN}
\left\lan \widetilde{F},\mathbf{N}\right\ran=\left(f-\sum_k\lan F,\p_kF\ran(\p_kf)\right)\prod_k(1-h_{kk}f)+Q_0(x,f,\na_Cf),
\endaligned
\end{equation}
and
\begin{equation}\aligned\label{hFijN}
&\left\lan \p_{ij}^2\widetilde{F},\mathbf{N}\right\ran=\left(h_{ij}-h_{ii}h_{jj}\de_{ij}f+\p_{ij}^2f\right)\prod_k(1-h_{kk}f)+Q_{1ij}(x,f,\na_Cf),
\endaligned
\end{equation}
where $|Q_{0}(x,f,\na_Cf)|\le c|f|\cdot|\na_Cf|$, $|Q_{1ij}(x,f,\na_Cf)|\le c|x_0|^{-1}|\na_Cf|^2+c|x_0|^{-2}|f|\cdot|\na_Cf|$.
Let $g=g_{ij}dx_idx_j$ be the metric of $M$ with 1-form $dx_i$ on $U_0$. Then
\begin{equation*}\aligned
g_{ij}=&\lan\p_i\widetilde{F},\p_j\widetilde{F}\ran=(1-h_{ii}f)(1-h_{jj}f)\de_{ij}+(\p_if)(\p_jf)\\
=&\f{\de_{ij}}{1+2h_{ii}f}+\left(1-\f{4}{1+2h_{ii}f}\right)h^2_{ii}f^2\de_{ij}+(\p_if)(\p_jf).
\endaligned
\end{equation*}
Hence
\begin{equation}\aligned\label{gupijQ2ij}
g^{ij}=&\de_{ij}(1+2h_{ii}f)+Q_{2ij}(x,f,\na_Cf),
\endaligned
\end{equation}
where $|Q_{2ij}(x,f,\na_Cf)|\le c|x_0|^{-2}f^2+c|\na_Cf|^2$ by \eqref{naCf518}. Since $\mathrm{spt}T\setminus B_{R_1}$ is a smooth self-expander, then
\begin{equation}\aligned
\left\lan\De_T\widetilde{F},\mathbf{N}\right\ran=\f12\left\lan \widetilde{F},\mathbf{N}\right\ran.
\endaligned
\end{equation}
Namely,
\begin{equation}\aligned\label{gupijhFijN}
\sum_{i,j}g^{ij}\left\lan\p_{ij}^2\widetilde{F},\mathbf{N}\right\ran=\f12\left\lan \widetilde{F},\mathbf{N}\right\ran.
\endaligned
\end{equation}
Combining \eqref{hFN}-\eqref{gupijQ2ij},\eqref{gupijhFijN}, we obtain
\begin{equation}\aligned
\sum_{k}\p_{kk}^2f+\f12\sum_k\lan F,\p_kF\ran(\p_kf)+\left(|A_C|^2-\f12\right)f+H_C+Q_3(x,f,\na_Cf,\na_C^2f)=0,
\endaligned
\end{equation}
and
\begin{equation}\aligned\label{5.44Q3}
|Q_3(x,f,\na_Cf,\na_C^2f)|\le& c\big(Q_0(x,f,\na_Cf)+Q_{1ij}(x,f,\na_Cf)\\
&+Q_{2ij}(x,f,\na_Cf)\left(|\na_C^2f|+|x_0|^{-1}\right)+|x_0|^{-1}f|\na_C^2f|\big)\\
\le&c\left(Q_0(x,f,\na_Cf)+Q_{1ij}(x,f,\na_Cf)+|x_0|^{-2}|f|\right)\\
\le&c|x_0|^{-3}\left(|x_0|\cdot|f|+|\na_Cf|\right),
\endaligned
\end{equation}
where we have used \eqref{na01fep2} and \eqref{naCf518} in the above inequality.
By
\begin{equation}\aligned
\sum_{k}\p_{kk}^2f(p)=&\sum_{k}\p_{kk}^2f(F(p))=\sum_k\p_k\left\lan\na_Cf,\p_kF\right\ran\\
=&\De_Cf+\sum_k\left\lan\na_Cf,\p_{kk}^2F\right\ran=\De_Cf+\sum_k\left\lan\na_Cf,h_{kk}\mathbf{n}\right\ran=\De_Cf,
\endaligned
\end{equation}
we complete the proof.
\end{proof}

Analog to the computation in the proof of Theorem \ref{LCDeCACQf}, from \eqref{YStf0}-\eqref{YStf2} we get the following corollary.
\begin{corollary}\label{sosse}
Let $S$ be a scaling $n$-dimensional smooth self-expander satisfying \eqref{Sp} for some constant $p\in(0,1]$ outside a compact set.
If a self-expander $T$ can be represented as a smooth graph over $S$ with the smooth graphic function $f$ outside a compact set, and with the same $C^{3,\a}$-regular tangent cone $C$ at infinity as $S$. Then there is a constant $\tilde{r}_{_S}>r_{_S}>0$, such that at $x\in S\setminus B_{\tilde{r}_{_S}}$,
\begin{equation}\aligned
L_Sf+\left(1-\f1{2p}\right)H_{S}=Q_{S}(x,f,\na_Sf,\na^2_Sf),
\endaligned
\end{equation}
where the function $Q_{S}(x,f,\na_Sf,\na^2_Sf)$ satisfies
\begin{equation}\aligned
|Q_{S}(x,f,\na_Sf,\na^2_Sf)|\le c|x|^{-3}\left(|x|\cdot|f|+|\na_Sf|\right).
\endaligned
\end{equation}
\end{corollary}

Now we consider an operator $L_0$ defined by
\begin{equation}\aligned\label{L0www}
L_0w=r^{1-n}e^{-\f{r^2}4}\f{\p}{\p r}\left(e^{\f{r^2}4}r^{n-1}\f{\p w}{\p r}\right)-\f w2.
\endaligned
\end{equation}
For any $\tau\in\R$ and $r>0$, we have
\begin{equation}\aligned\label{Lsnepr2}
L_0\left(r^{-n-1+\tau}e^{-\f{r^2}4}\right)=\left(-\f{\tau}2+\f{3(n+1)-(n+4)\tau+\tau^2}{r^2}\right)r^{-n-1+\tau}e^{-\f{r^2}4}.
\endaligned
\end{equation}
Then for any $s\in\R$
\begin{equation}\aligned\label{L0rnrn2}
L_0\left(t(r+s)r^{-n-2}e^{-\f{r^2}4}\right)=tr^{-n-2}e^{-\f{r^2}4}\left(\f s2+\f{3n+3}{r}+\f{4n+8}{r^2}s\right).
\endaligned
\end{equation}

Let $\Si$ be an $(n-1)$-dimensional embedded $C^{3,\a}$-regular hypersurface in $\S^n$ and $C=C\Si$ is a cone over $\Si$. Let $A_\Si$ denote the second fundamental form of $\Si$. Denote $x=r\xi\in C\Si\setminus\{0\}$ with $r>0$ and $\xi\in\Si$. For any $\psi\in C^2(\Si\times\R^+)=C^2(C\Si\setminus\{0\})$, by \eqref{LSFSE} we have
\begin{equation}\aligned\label{LCDeSiLsf}
L_C\psi=L_0\psi+\f1{r^2}\left(\De_\Si+|A_\Si|^2\right)\psi.
\endaligned
\end{equation}
Clearly, $\De_\Si+|A_\Si|^2$ is an elliptic operator with eigenfunctions $\{\varphi_k\}_{k=1}^\infty$ (with $\varphi_1>0$) and corresponding eigenvalues $\la_1<\la_2\le\la_3\le\cdots\le\la_k\cdots<\infty$. Namely,
\begin{equation}\aligned\label{DeSivaiphik}
\De_\Si\varphi_k+|A_\Si|^2\varphi_k+\la_k\varphi_k=0.
\endaligned
\end{equation}
We assume $\int_\Si|\varphi_k|^2=1$ for each $k$, then $\{\varphi_k\}_{k=1}^\infty$ forms an orthonormal basis of $L^2(\Si)$.

\subsection{Asymptotic decay estimates for self-expanders}
Now let us study the asymptotic decay estimates for self-expanders which converge to the mean convex cones by the Jacobi field operator on self-expanders at infinity.
Recall the definition for the mean convex cone in section 1 of this paper.
\begin{lemma}\label{lowgse}
Let $C$ be an $n$-dimensional $C^{3,\a}$-regular mean convex cone pointing into $\Om$ in $\R^{n+1}$ with $\p\Om=C$. If outside a compact set, $M$ is a smooth self-expander in $\Om$, which is a graph over $C$ with the positive graphic function $f$. Then there are constants $\ep_2>0$ and $R_5>0$ such that on $C\setminus B_{R_5}$ we have
\begin{equation}\aligned
f(\xi r)\ge\ep_2r^{-n-1}e^{-\f{r^2}4}\varphi_1(\xi).
\endaligned
\end{equation}
\end{lemma}
\begin{proof}
Combining \eqref{L0rnrn2} \eqref{LCDeSiLsf} and \eqref{DeSivaiphik}, we obtain
\begin{equation}\aligned
L_C\left(t(r+s)r^{-n-2}e^{-\f{r^2}4}\varphi_1\right)=\left(\f s2+\f{3n+3}{r}+\f{4n+8}{r^2}s-\f{r+s}{r^2}\la_1\right)tr^{-n-2}e^{-\f{r^2}4}\varphi_1.
\endaligned
\end{equation}
There is a constant $R_5>R_4$ depending only on the cone $C$, such that on $C\setminus B_{R_5}$, we have
\begin{equation}\aligned\label{LCf135.57}
L_C\left(t(r+1)r^{-n-2}e^{-\f{r^2}4}\varphi_1\right)&>\f13tr^{-n-2}e^{-\f{r^2}4}\varphi_1
\endaligned
\end{equation}
for any $t>0$. Since $f$ is positive, then there is a sufficiently small constant $\ep_{R_5}>0$ such that
$$\inf_{\p B_{R_5}}\left(f-\ep_{R_5}(R_5+1)R_5^{-n-2}e^{-\f{R_5^2}4}\varphi_1\right)\ge0.$$
Denote $\phi=\ep_{R_5}(r+1)r^{-n-2}e^{-\f{r^2}4}\varphi_1$ for convenience. Suppose that $R_5$ is sufficiently large such that $|A_C|^2<\f1{100}$ on $C\setminus B_{R_5}$.
We assume that there is a point $z=(p,r)\in\Si\times(R_5,\infty)$ such that
\begin{equation}\aligned\label{fz5.58CR4}
f(z)-\phi(z)=\inf_{C\setminus B_{R_4}}\left(f-\phi\right)<0.
\endaligned
\end{equation}
Then $\na_Cf=\na_C\phi$, and $\De_C(f-\phi)\ge0$ at $z$. However, combining $H_C\ge0$, \eqref{LCDeCACQf} and \eqref{LCf135.57} at $z$ we have
\begin{equation}\aligned
0<\left(\f1{100}-\f12\right)(f-\phi)\le&\De_C(f-\phi)+\f12\lan z,\na_C(f-\phi)\ran+\left(|A_C|^2-\f12\right)(f-\phi)\\
=&L_Cf-L_C\phi<Q_C(z,f,\na_Cf,\na^2_Cf)-H_C-\f{\phi}{3(|z|+1)}\\
\le&c\left(\f{f}{|z|^2}+\f{|\na_Cf|}{|z|^3}\right)-\f{\phi}{3(|z|+1)}\\
\le&c\left(\f{\phi}{|z|^2}+\f{|\na_C\phi|}{|z|^3}\right)-\f{\phi}{3(|z|+1)},
\endaligned
\end{equation}
which is impossible for the sufficiently large $R_5$ since $\phi$ is a smooth positive function. Hence \eqref{fz5.58CR4} does not hold, i.e.,
\begin{equation}\aligned
\inf_{C\setminus B_{R_5}}\left(f-\phi\right)\ge0,
\endaligned
\end{equation}
and then we complete the proof.
\end{proof}

\begin{lemma}\label{upgse}
Let $S$ be a scaling smooth self-expander satisfying \eqref{Sp} with a constant $p\in(0,1]$ outside a compact set. If a self-expander $M$ can be represented as a smooth graph over $S$ with the graphic function $\widetilde{f}$, and $M$ has the same $C^{3,\a}$-regular tangent cone $C$ at infinity as $S$. Denote $\pi_S$ be the projection from $C$ to $S$. Then there are constants $R_{6}>0$ independent of $p$, and $\ep_{3}>0$, such that if $\widetilde{f}\circ\pi_S$ is positive on $C\cap B_{R_6}$, then for any $\xi r\in C\setminus B_{R_{6}}$ we have
\begin{equation}\aligned
\widetilde{f}\circ\pi_S(\xi r)\le\ep_{3}^{-1}\ r^{-n-1}e^{-\f{r^2}4}\varphi_1(\xi)\qquad if\ \ \left(\f1{2p}-1\right)H_S\ge0,
\endaligned
\end{equation}
and
\begin{equation}\aligned
\widetilde{f}\circ\pi_S(\xi r)\ge\ep_{3}\ r^{-n-1}e^{-\f{r^2}4}\varphi_1(\xi)\qquad if\ \ \left(\f1{2p}-1\right)H_S\le0.
\endaligned
\end{equation}
\end{lemma}
\begin{proof}
Let $C=C\Si$ be the regular tangent cone of $S$ at infinity, with the metric $g_C$ and the Levi-Civita connection $\na_C$ outside the vertex, respectively. Up to a compact set, $S$ can be seen as a graph over $C$ with a graphic function $f$. Let $\{\xi_i\}$ be a curve coordinate of $C$ and $g_C=\si_{ij}d\xi_id\xi_j$, then the metric of $S$ can be denoted by
$$\tilde{g}=\tilde{g}_{ij}d\xi_id\xi_j=\left(\si_{ij}+\widetilde{D}_if\widetilde{D}_jf\right)d\xi_id\xi_j,$$
where $\widetilde{D}=\na_C$, and $\widetilde{D}_if=\lan \widetilde{D}f,\f{\p}{\p\xi_i}\ran$ for convenience.
Moreover,
$$\tilde{g}^{ij}=\si^{ij}-\f{\widetilde{D}^if\widetilde{D}^jf}{1+|\widetilde{D}f|^2},\quad \det\tilde{g}_{ij}=(1+|\widetilde{D}f|^2)\det\si_{ij},$$
where $\widetilde{D}^if=\si^{ij}\widetilde{D}_jf$ and $|\widetilde{D}u|^2=\si^{ij}\widetilde{D}_if\widetilde{D}_jf$. Let $\nu_C$ be the unit normal vector of $C$. There is a sufficiently large constant $R_{6}>0$ such that for each $X\in S\setminus B_{R_{6}}$, there is a unique $x\in C$ with $X=x+f(x)\nu_C$. Hence for any $\mathcal{F}\in C^2(S)$ and $X=x+f(x)\nu_C\in S$, we identify $\mathcal{F}(X)$ with $\mathcal{F}(x)$. Then it follows that (see formula (2.2) in \cite{DJX} for instance)
\begin{equation}\aligned
\De_S \mathcal{F}=\f1{\sqrt{\tilde{g}_{kl}}}\p_{\xi_j}\left(\sqrt{\tilde{g}_{kl}}\tilde{g}^{ij}\p_{\xi_j}\mathcal{F}\right)=\De_C \mathcal{F}+P_1(X,\na_C\mathcal{F}),
\endaligned
\end{equation}
where $|P_1(X,\na_C\mathcal{F})|\le c|X|^{-5}|\na_C\mathcal{F}|$. Since
$$\nu\triangleq\f1{\sqrt{1+|\na_Cf|^2}}\left(-\sum_i\widetilde{D}^if\f{\p}{\p\xi_i}+\nu_C\right)$$
is the unit normal vector of $S$, one gets
\begin{equation}\aligned
\na_S\mathcal{F}=&\sum_i\widetilde{D}^i\mathcal{F}\f{\p}{\p\xi_i}-\left\lan\sum_i\widetilde{D}^i\mathcal{F}\f{\p}{\p\xi_i},\nu\right\ran\nu
=\sum_i\widetilde{D}^i\mathcal{F}\f{\p}{\p\xi_i}+\f{\lan\na_Cf,\na_C\mathcal{F}\ran}{\sqrt{1+|\na_Cf|^2}}\nu\\
=&\sum_i\left(\widetilde{D}^i\mathcal{F}-\f{\lan\na_Cf,\na_C\mathcal{F}\ran}{1+|\na_Cf|^2}\widetilde{D}^if\right)\f{\p}{\p\xi_i}
+\f{\lan\na_Cf,\na_C\mathcal{F}\ran}{1+|\na_Cf|^2}\nu_C\\
=&\na_C\mathcal{F}-\f{\lan\na_Cf,\na_C\mathcal{F}\ran}{1+|\na_Cf|^2}\na_Cf+\f{\lan\na_Cf,\na_C\mathcal{F}\ran}{1+|\na_Cf|^2}\nu_C,
\endaligned
\end{equation}
and
\begin{equation}\aligned
\lan X,\na_S\mathcal{F}\ran=\lan x,\na_C\mathcal{F}\ran+\left(f-\lan x,\na_Cf\ran\right)\f{\lan\na_Cf,\na_C\mathcal{F}\ran}{1+|\na_Cf|^2}=\lan x,\na_C\mathcal{F}\ran+P_2(x,\na_C\mathcal{F}),
\endaligned
\end{equation}
with $|P_2(x,\na_C\mathcal{F})|\le c|x|^{-3}|\na_C\mathcal{F}|$.

Now we identify $\widetilde{f}(X)$ with $\widetilde{f}(x)$ for $X=x+f(x)\nu_C\in S$. From \eqref{nakA}, we have $|A_S|^2\le c|x|^{-2}$ at the point $X$. For any $x\in S\setminus B_{R_{6}}$, one has
\begin{equation}\aligned\label{LsftildePtilde}
L_S \widetilde{f}=\De_C \widetilde{f}+\f12\lan x,\na_C \widetilde{f}\ran+\left(|A_C|^2-\f12\right)\widetilde{f}+\widetilde{P}(x,\widetilde{f},\na_C \widetilde{f},\na^2_C \widetilde{f}),
\endaligned
\end{equation}
where the function $\widetilde{P}$ satisfies
\begin{equation}\aligned
\left|\widetilde{P}(x,\widetilde{f},\na_C \widetilde{f},\na^2_C \widetilde{f})\right|\le c|x|^{-2}\left(|\widetilde{f}|+|x|^{-1}|\na_C \widetilde{f}|\right).
\endaligned
\end{equation}
For the sufficiently large $R_{6}>0$, from \eqref{L0rnrn2} one has
\begin{equation}\aligned\label{LCt-sn+2}
L_C\left(t(r-s)r^{-n-2}e^{-\f{r^2}4 }\varphi_1\right)<-\f{t}3sr^{-n-2}e^{-\f{r^2}4}\varphi_1
\endaligned
\end{equation}
on $C\setminus B_{R_{6}}$ for any constants $t>0$ and $s\ge1$.
There is a constant $\be_{R_6}>0$ such that
$$\sup_{\p B_{R_6}}\left(\widetilde{f}-\be_{R_6}(r-1)r^{-n-2}e^{-\f{r^2}4}\varphi_1\right)=0.$$
Denote $\psi=\be_{R_{6}}(r-1)r^{-n-2}e^{-\f{r^2}4}\varphi_1$ for convenience. Suppose $|A_C|^2<\f1{100}$ on $S\setminus B_{R_{6}}$ for the sufficiently large constant $R_{6}$.

When $\left(\f1{2p}-1\right)H_S\ge0$, we claim
\begin{equation}\aligned\label{5.69tfpsi}
\sup_{z\in C\setminus B_{R_{6}}}\left(\widetilde{f}(z)-\psi(z)\right)\le0.
\endaligned
\end{equation}
If not, there is a point $z\in C\setminus \overline{B_{R_{6}}}$ such that $\widetilde{f}(z)-\psi(z)=\sup_{z\in C\setminus B_{R_{6}}}\left(\widetilde{f}-\psi\right)>0,$
then $\na_C\widetilde{f}=\na_C\psi$, and $\De_C(\widetilde{f}-\psi)\le0$ at $z$. However, from \eqref{LsftildePtilde}\eqref{LCt-sn+2} and Corollary \ref{sosse}, at $z$ we have
\begin{equation}\aligned\label{tfpsi568}
&\left(\f1{100}-\f12\right)(\widetilde{f}-\psi)\ge\De_C(\widetilde{f}-\psi)+\f12\lan z,\na_C(\widetilde{f}-\psi)\ran+\left(|A_C|^2-\f12\right)(\widetilde{f}-\psi)\\
=&L_C\widetilde{f}-L_C\psi
\ge L_S\widetilde{f}-\widetilde{P}(z,\widetilde{f},\na_C\widetilde{f},\na^2_C\widetilde{f})+\f{1}{3(|z|-1)}\psi\\
\ge&\left(\f1{2p}-1\right)H_S+Q_S(z,\widetilde{f},\na_S\widetilde{f},\na_S^2\widetilde{f})-
\widetilde{P}(z,\widetilde{f},\na_C\widetilde{f},\na^2_C\widetilde{f})+\f{1}{3(|z|-1)}\psi\\
\ge& -c|z|^{-2}\left(|\widetilde{f}|+|z|^{-1}|\na_C \widetilde{f}|\right)+\f{1}{3(|z|-1)}\psi\\
=& -c|z|^{-2}(\widetilde{f}-\psi)-c|z|^{-2}\left(\psi+|z|^{-1}|\na_C\psi|\right)+\f{1}{3(|z|-1)}\psi,
\endaligned
\end{equation}
which is in contradiction to the sufficiently large constant $R_{6}$. In particular, $R_6$ is independent of $p\in(0,1]$. Hence we obtain the claim \eqref{5.69tfpsi}.

Similarly, for $\left(\f1{2p}-1\right)H_S\le0$ we can show $\widetilde{f}\ge\ep_{3}\ r^{-n-1}e^{-\f{r^2}4}\varphi_1$ if $\ep_3$ is sufficiently small and $R_6$ is sufficiently large.
\end{proof}

Combining Lemma \ref{lowgse} and Lemma \ref{upgse}, we get the following corollary.
\begin{corollary}\label{updownC}
Let $C$ be an $n$-dimensional regular, minimal but not minimizing cone in $\R^{n+1}$. If outside a compact set, $M$ is a self-expander in one of components of $\R^{n+1}\setminus C$, which is a graph over $C$ with the positive graphic function $f$. Then there are constants $\ep_4>0$ and $R_7>0$ such that for each $\xi r\in C\setminus B_{R_7}$ we have
\begin{equation}\aligned\label{asymse}
\ep_4r^{-n-1}e^{-\f{r^2}4}\varphi_1(\xi)\le f(\xi r)\le\ep_4^{-1}r^{-n-1}e^{-\f{r^2}4}\varphi_1(\xi).
\endaligned
\end{equation}
\end{corollary}

Recall that $\Si$ is an $(n-1)$-dimensional embedded $C^{3,\a}$-regular manifold in $\S^n$ with positive mean curvature $H_\Si$. For all $x\in C\setminus \{0\}$, by the definition of $L_C$ in \eqref{LCDeSiLsf} we have
\begin{equation}\aligned
L_Cf=&\f{\p^2f}{\p r^2}+\f{n-1}r\f{\p f}{\p r}+\f r2\f{\p f}{\p r}-\f f2+\f1{r^2}\left(\De_\Si f+|A_\Si|^2f\right).
\endaligned
\end{equation}
Set $$\z_\a=\f{H_\Si}r+\f{\a}{r^3}$$
for each constant $\a\in\R$. Note that the mean curvature of $C$: $H_C=\f1{|x|}H_\Si$ at $x\in C\setminus\{0\}$, then
\begin{equation}\aligned
L_C\z_\a+H_C=&\f{1}{r^3}\left((|A_\Si|^2+3-n)H_\Si+\De H_\Si-2\a\right)+\f{|A_\Si|^2+3(5-n)}{r^5}\a.
\endaligned
\end{equation}
So there exists a constant $\a_0>0$ depending only on $n$ and the cone $C$, such that for any $|\a|\ge\a_0$ we have
\begin{equation}\label{LCphiaHCr-3}
L_C\z_\a+H_C\left\{\begin{split}
>&-\f{\a}{r^3},\qquad \a<0\\
<&-\f{\a}{r^3},\qquad \a>0.\\
\end{split}\right.
\end{equation}

\begin{lemma}\label{AsmPMC}
Let $C$ be an $n$-dimensional $C^{3,\a}$-regular cone in $\R^{n+1}$ with positive mean curvature pointing into $\Om$, which is one of components of $\R^{n+1}\setminus C$. If outside a compact set, $M$ is a smooth self-expander in $\Om$, which is a graph over $C$ with the graphic function $f$. Then there are constants $K_C>0$ and $R_8>0$ such that we have
\begin{equation}\aligned
\f{H_\Si(\xi)}r-\f{K_C}{r^3}\le f(\xi r)\le \f{H_\Si(\xi)}r+\f{K_C}{r^3}\qquad \mathrm{for\ any\ } \xi r\in C\setminus B_{R_8},
\endaligned
\end{equation}
where $H_\Si(\xi)$ is the mean curvature of $\Si$ at $\xi$.
\end{lemma}
\begin{proof}
Suppose that $|A_C|^2<\f1{100}$ on $C\setminus B_{R_{8}}$ for a sufficiently large constant $R_{8}\ge R_4$.
Let $\be$ be a constant depending on $R_8$ such that $\be\le K_0\triangleq-\max\{c(\sup_{\Si}H_\Si+1),\a_0\}$ and
$\inf_{\p B_{R_8}}\left(f-\z_{\be}\right)\ge0.$
We claim
\begin{equation}\aligned\label{5.76fzbe}
\inf_{z\in C\setminus B_{R_8}}\left(f(z)-\z_{\be}(z)\right)\ge0
\endaligned
\end{equation}
for the sufficiently large $R_8>0$.
If not, there is a point $z=\xi r\in C\setminus B_{R_8}$ such that
\begin{equation}\aligned
f(z)-\z_{\be}(z)=\inf_{C\setminus B_{R_8}}\left(f-\z_{\be}\right)<0,
\endaligned
\end{equation}
then $\na_Cf=\na_C\z_{\be}$, and $\De_C(f-\z_{\be})\ge0$ at $z$. Note that $f$ is positive and $|z|=r>R_8$, then by \eqref{LCDeCACQf} and \eqref{LCphiaHCr-3} at $z$ we get
\begin{equation}\aligned
\left(\f1{100}-\f12\right)(f-\z_{\be})\le&\De_C(f-\z_{\be})+\f12\lan z,\na_C(f-\z_{\be})\ran+\left(|A_C|^2-\f12\right)(f-\z_{\be})\\
=&L_Cf-L_C\z_{\be}=Q(z,f,\na_Cf,\na^2_Cf)-L_C\z_{\be}-H_C\\
\le&\f{c}{|z|^{3}}\left(|z|\cdot|f|+|\na_C f|\right)+\f {\be}{|z|^3}\\
\le&\f{c}{|z|^{3}}\left(|z|\cdot\left|f-\z_{\be}\right|+|z|\cdot\left|\z_{\be}\right|+|\na_C\z_{\be}|\right)+\f {\be}{|z|^3},
\endaligned
\end{equation}
which implies
\begin{equation}\aligned
0\le\f{c}{|z|^{3}}\left(|z|\cdot\left|\z_{\be}\right|+|\na_C\z_{\be}|\right)+\f {\be}{|z|^3}
\endaligned
\end{equation}
for the sufficiently large $R_8>0$. By the definition of $\z_{\be}$ we get
\begin{equation}\aligned\label{HSiKr23}
0<\f{c}{|z|^{3}}\left(H_\Si-\f {\be}{|z|^2}+\left|\na_C\left(\f{H_\Si}{|z|}\right)\right|-\f{4{\be}}{|z|^4}\right)+\f {\be}{|z|^3}.
\endaligned
\end{equation}
For the sufficiently large $R_8>0$, \eqref{HSiKr23} does not hold for $\be\le K_0$ and $r\ge R_8$. Hence we get the claim \eqref{5.76fzbe}.

There is a constant $\La\ge-K_0>0$ depending on $R_8$ such that
$\sup_{\p B_{R_8}}\left(f-\z_{\La}\right)\le0.$
We claim
\begin{equation}\aligned\label{5.81fLa}
\sup_{z\in C\setminus B_{R_8}}(f(z)-\z_\La(z))\le0,
\endaligned
\end{equation}
for the sufficiently large $R_8>0$.
If not, there is a point $z'\in C\setminus \overline{B_{R_8}}$ such that
\begin{equation}\aligned
f(z')-\z_\La(z')=\sup_{C\setminus B_{R_8}}\left(f-\z_\La\right)>0,
\endaligned
\end{equation}
then $\na_Cf=\na_C\z_\La$, and $\De_C(f-\z_\La)\le0$ at $z'$. However, by \eqref{LCDeCACQf} and \eqref{LCphiaHCr-3} at $z'$
\begin{equation}\aligned
\left(\f1{100}-\f12\right)(f-\z_\La)\ge&\De_C(f-\z_\La)+\f12\lan z',\na_C(f-\z_\La)\ran+\left(|A_C|^2-\f12\right)(f-\z_\La)\\
=&L_Cf-L_C\z_\La=Q(z',f,\na_Cf,\na^2_Cf)-L_C\z_\La-H_C\\
\ge&-\f{c}{|z'|^{3}}\left(|z'|f+|\na_Cf|\right)+\f{\La}{|z'|^3}\\
=&-\f{c}{|z'|^2}(f-\z_\La)-\f{c}{|z'|^{3}}\left(|z'|\z_\La+|\na_C\z_\La|\right)+\f{\La}{|z'|^3}.
\endaligned
\end{equation}
For the sufficiently large $R_8>0$ we have
\begin{equation}\aligned
0>&-\f{c}{|z'|^{3}}\left(H_\Si+\f{\La}{|z'|^2}+\left|\na_C\left(\f{H_\Si}{|z'|}\right)\right|+\f{4\La}{|z'|^4}\right)+\f{\La}{|z'|^3}.
\endaligned
\end{equation}
The above inequality does not hold for any $\La\ge R_8$ and $|z'|\ge R_8$ if $R_8$ is sufficiently large. Therefore, \eqref{5.81fLa} holds,
and we complete the proof.
\end{proof}

\section{Existence for self-expanders}

By the dimension estimates of singular sets for minimizing currents in codimension 1 (see Theorem 37.7 in \cite{S} for example),
if $T$ is an \textbf{E}-minimizing self-expanding current in $\R^{n+1}$, then $\mathcal{H}^{s}(\mathrm{Sing}T)=0$ for any $s\ge\max\{0,n-7+\a\}$ with $\a>0$ because $T$ is a minimizing current in $\left(\R^{n+1},e^{\f{|X|^2}{2n}}\sum_{i=1}^{n+1}dx_i^2\right)$ with $|X|^2=\sum_{i=1}^{n+1}x_i^2$.
There are quite a few of works on maximum principle for singular minimal hypersurfaces in smooth manifolds.
For convenience, we need the following maximum principle obtained by Solomon-White in \cite{SW}.
\begin{lemma}\label{mpbySolomon-W}
Let $\Om$ be a domain in a space $N$ with the smooth boundary $\p\Om$, and $T$ be a stationary varifold in $\overline{\Om}$, where $N=\left(\R^{n+1},e^{\g|X|^2}\sum_{i=1}^{n+1}dx_i^2\right)$ for a constant $\g\ge0$. If there is a point $x\in \mathrm{spt}T\cap\p\Om\setminus\p(\mathrm{spt}T)$, and $\p\Om$ has nonnegative mean curvature pointing into $\Om$ in a neighborhood $O_x$ of $x$, then $\mathrm{spt}T\cap O_x =\p\Om\cap O_x$.
\end{lemma}
Let us recall a sharp strong maximum principle for singular minimal hypersurfaces in manifolds, recently showed by N. Wickramasekera \cite{Wi13} (the ones in \cite{Wi} or \cite{Il} are also sufficient for $\mathbf{E}$-minimizing cases by dimension estimates). In particular, we can choose the ambient manifold to be  $\left(\R^{n+1},e^{\f{|X|^2}{2n}}\sum_{i=1}^{n+1}dx_i^2\right)$.
\begin{lemma}\label{msem}
Let $C_V$ be a stationary cone in $\R^{n+1}$ with integer multiplicity and with singular set $\mathrm{Sing}C_V$ satisfying $\mathcal{H}^{n-1}(\mathrm{Sing}C_V)=0$, and $\Om$ be a domain with $\p\Om=\mathrm{spt}C_V$. If $T$ is an $n$-dimensional integer varifold self-expander in $\overline{\Om}$ with $\p T=0$, then either $\mathrm{spt}T\cap \mathrm{spt}C_V=\emptyset$ or $\mathrm{spt}T=\mathrm{spt}C_V$.
\end{lemma}

\subsection{Existence for \textbf{E}-minimizing self-expanding currents}
For any current $S$ in $\R^{n+1}$, let $S\setminus K=S\llcorner\left(\R^{n+1}\setminus K\right)$ for any set $K\in\R^{n+1}$.
Here, the notation '$\llcorner$' can be found in Chapter 26 in \cite{S}.

Ilmanen \cite{Il1} proved the existence of self-expanding hypersurfaces with prescribed tangent cones at infinity. For completeness, we give another proof for $C^2$-regular cones by constructing a family of barrier functions.
\begin{theorem}\label{WeakexistSE}
Let $C$ be a cone over an embedded $C^2$-regular hypersurface in $\S^n$, then there is a multiplicity one \textbf{E}-minimizing self-expanding current $T$ in $\R^{n+1}$ with $\p T=0$
such that the tangent cone of $T$ at infinity is $\llbracket C\rrbracket$. 
\end{theorem}
\begin{proof}
By the assumption, principal curvature of $\Si\triangleq C\cap\S^{n}$ are uniformly bounded.
Then there is a uniform positive constant $0<\k<\f12$ depending only on $C$ so that for any $y\in C\setminus\{0\}$, we have two cones $C_{\k,\pm}$ centered at the origin isometric to $\left\{\left(x,\f{|x|}\k\right)\big| x\in\R^n\right\}$ such that $C_{\k,\pm}\cap C$ are the same radial line through $y$. Obviously, there are two unique points $y_\pm$ related to $y$ with $|y_\pm|=|y|$ such that $C_{\k,\pm}\subset\{X\in\R^{n+1}|\,\lan X,y_\pm\ran>0\}$ are rotational symmetric w.r.t. the lines $\{ty_\pm|\ t\in\R\}$, respectively.

By compactness theorem for currents (see \cite{LY}\cite{S} for instance),
there exists an open set $E_r\subset\R^{n+1}$ for any $r>0$ such that $\p E_r\setminus B_r=C\setminus B_r$ and
\begin{equation}\aligned
\int_{B_r}|D\chi_{_{E_r}}|e^{\f{|X|^2}4}dX=\inf\left\{\int_{B_r}|D\chi_{_U}|e^{\f{|X|^2}4}dX\Big|\ \p U\setminus B_r=C\setminus B_r\right\}.
\endaligned
\end{equation}
Hence $\p \llbracket E_r\rrbracket\llcorner B_r$ is an $n$-dimensional multiplicity one \textbf{E}-minimizing self-expanding current in $B_r$.
Then the singular set of $\p \llbracket E_r\rrbracket\llcorner B_r$ has Hausdorff dimension $\le n-7$ in case $n\ge7$ and is empty in case $n\le6$.
Let $P_{y_\pm}$ be the $n$-dimensional hyperplanes through the origin which is perpendicular to the vectors $y_\pm$, respectively.
Let $\phi_\la$ be a smooth solution to \eqref{Ju} with $\phi_\la'\ge \phi_\la'(0)=0$ and $\lim_{r\rightarrow\infty}\f{\phi_\la(r)}r=\la>0$.
For $\la\ge\f1\k$, we set $\mathcal{S}^\la_{y_\pm}$ being a rigid motion of a self expander defined by
$$\mathcal{S}^\la_{y_\pm}=\left\{z+ty_\pm\in\R^{n+1}\Big|\, z\in P_{y_\pm},\, t=\phi_{\la}(z)\right\}.$$
From Lemma \ref{phila0}, we consider the infimum of $\la$ such that $\p E_r\cap B_r\cap\left(\mathcal{S}^\la_{y_+}\cup \mathcal{S}^\la_{y_-}\right)=\emptyset$. We conclude that $\p E_r\cap B_r\cap\left(\mathcal{S}_{y_+}\cup \mathcal{S}_{y_-}\right)=\emptyset$ by maximum principle (Lemma \ref{mpbySolomon-W}), where $\mathcal{S}_{y_\pm}=\mathcal{S}^{1/\k}_{y_\pm}$. Let $\mathcal{E}_C=\bigcup_{y\in\S^n\cap C}\left(\mathcal{S}_{y_+}\cup \mathcal{S}_{y_-}\right)$, then we have
\begin{equation}\aligned\label{6ErBrEc}
\p E_r\cap B_r\cap \mathcal{E}_C=\emptyset.
\endaligned
\end{equation}
For any two sets $K_1,K_2$ in $\R^{n+1}$, we define their Hausdorff distance $d_H(K_1,K_2)$ by
\begin{equation*}\aligned
d_H(K_1,K_2)=&\max\left\{\sup_{x\in K_1}\mathrm{dist}(x,K_2),\sup_{x\in K_2}\mathrm{dist}(x,K_1)\right\}\\
=&\max\left\{\sup_{x\in K_1}\inf_{y\in K_2}|x-y|,\sup_{x\in K_2}\inf_{y\in K_1}|x-y|\right\}.
\endaligned
\end{equation*}
From the definition of $\mathcal{S}_{y_+}$ and \eqref{phicomp}, one has
\begin{equation}\aligned\label{6.2222}
\mathrm{dist}(y,\mathcal{S}_{y_+})\le\left(\f{n+1}{\k}+2\right)\min\{1,\k^{-1}|y|^{-1}\}
\endaligned
\end{equation}
for any $y\in C\setminus\{0\}$. In particular, $\mathrm{dist}(0,\mathcal{S}_{y_+})\le \f{n+1}{\k}+2$.
For any $y\in C$ with $|y|>\f{n+1}{\k}+2$, there is a point $y^*_+\subset\mathcal{S}_{y_+}\cap\p B_{|y|}$ such that
$$|y-y^*_+|=\mathrm{dist}(y,\mathcal{S}_{y_+}\cap\p B_{|y|}).$$
For any $y\in C$, the tangent space $T_yC$ of $C$ at $y$ is asymptotic to the tangent space $T_{y^*_+}\mathcal{S}_{y_+}$ of $\mathcal{S}_{y_+}$ at $y^*_+$ as $|y|\rightarrow\infty$. Let $c_\k>\f{n+1}{\k}+2$ be a sufficiently large constant, then from \eqref{6.2222} we have
$$|y-y_+^*|<2\ \mathrm{dist}(y,\mathcal{S}_{y_+})<\f{2(n+2)}{\k^2s}$$
for any $y\in C\cap\p B_s$ and $s\ge c_\k$.
From the definition of $\mathcal{E}_C$, we conclude
\begin{equation}\aligned\label{distpErpBs}
d_H(C\cap\p B_s,\mathcal{E}_C\cap\p B_s)<\f{2(n+2)}{\k^2s}
\endaligned
\end{equation}
for each $s\ge c_\k$.
Combining \eqref{6ErBrEc} and \eqref{distpErpBs}, we have
\begin{equation}\aligned\label{dpErpBsC}
d_H(\p E_r\cap\p B_s,\ C\cap\p B_s)\le \f{2(n+2)}{\k^2s}
\endaligned
\end{equation}
for every $s\in[c_\k,r]$. Let $M_r$ be the regular set of $\p \llbracket E_r\rrbracket\cap B_r$. By \eqref{laplace}, we have
\begin{equation}\aligned\label{6.62nHnMr}
2n\mathcal{H}^n(M_{r})\le\int_{M_{r}}\De_{M_r}|X|^2\le2\int_{\p M_{r}}|X|\le2r\mathcal{H}^{n-1}(\p M_{r}),
\endaligned
\end{equation}
where $\De_{M_r}$ denotes Laplacian of $M_r$.
Note that $\mathrm{dim}\left(\mathrm{Sing}\p \llbracket E_r\rrbracket\cap B_r\right)\le \max\{0,n-7\}$, then \eqref{6.62nHnMr} implies
\begin{equation}\aligned
\f{\mathcal{H}^n(\p E_{r}\cap B_r)}{r^n}\le\f{\mathcal{H}^{n-1}(\p E_{r}\cap\p B_r)}{nr^{n-1}}=\f{\mathcal{H}^{n-1}(C\cap\p B_r)}{nr^{n-1}}.
\endaligned
\end{equation}
From \eqref{monSEr-nBr}, it follows that
\begin{equation}\aligned\label{ErBrrn}
\f{\mathcal{H}^n(\p E_{\r}\cap B_\r)}{\r^n}\le\f{\mathcal{H}^{n-1}(\p E_{r}\cap\p B_r)}{nr^{n-1}}=\f{\mathcal{H}^{n-1}(C\cap\p B_r)}{nr^{n-1}}=\f1n\mathcal{H}^{n-1}(C\cap\p B_1)
\endaligned
\end{equation}
for all $0<\r\le r$.

By compactness theorem of currents (see \cite{LY} or \cite{S} for example), there is a sequence $r_i\rightarrow\infty$ such that $\p \llbracket E_{r_i}\rrbracket\llcorner B_{r_i}$ converges to an \textbf{E}-minimizing self-expanding current $T$, whose multiplicity is $k$ and support is $\p E_*$ for some open set $E_*$. According to \eqref{ErBrrn}, we have
\begin{equation}\aligned\label{kE*}
k\f{\mathcal{H}^n(\p E_*\cap B_\r)}{\r^n}\le\f1n\mathcal{H}^{n-1}(C\cap\p B_1)
\endaligned
\end{equation}
for all $\r>0$. From \eqref{dpErpBsC} it is clearly that
\begin{equation}\aligned
\lim_{r\rightarrow\infty}\f{\mathcal{H}^n(\p E_*\cap B_r)}{r^n}\ge\f1n\mathcal{H}^{n-1}(C\cap\p B_1).
\endaligned
\end{equation}
Thus the multiplicity $k=1$ and $T$ is a multiplicity one \textbf{E}-minimizing self-expanding current in $\R^{n+1}$ with $\p T=0$. Moreover, \eqref{kE*} implies that the tangent cone of $T$ at infinity is $\llbracket C\rrbracket$.
\end{proof}
From \eqref{dpErpBsC}, one has
\begin{equation}\aligned\label{dTsBsC}
\limsup_{s\rightarrow\infty}\big(sd_H(\mathrm{spt}\p T\cap\p B_s,\ C\cap\p B_s)\big)\le \f{2(n+2)}{\k^2},
\endaligned
\end{equation}
where $T,\k$ are as in the above Theorem.

\subsection{Area-minimizing cones and \textbf{E}-minimizing self-expanding currents}
In Theorem \ref{WeakexistSE}, for the \textbf{E}-minimizing self-expanding current $T$, $\mathrm{spt}T$ may be equal to the cone $C$ itself.
In fact, we can not find such $T$ when $C$ is area-minimizing.

\begin{lemma}\label{ammse}
Any non area-minimizing hypercone cannot be an \textbf{E}-minimizing self-expanding hypersurface. Moreover, if a minimal cone $C$ is an \textbf{E}-minimizing self-expanding hypersurface in $\overline{\Om}$, then it is an area-minimizing cone in $\overline{\Om}$, where $\Om$ is a component of $\R^{n+1}\setminus C$.
\end{lemma}
\begin{proof}
Let $C$ be a minimal hypercone with possible singularities.  
If $C$ is also an \textbf{E}-minimizing self-expanding hypersurface in $\R^{n+1}$, then for any fixed $r>0$ and each $t>0$
\begin{equation}\aligned
\int_0^{r}\mathcal{H}^{n-1}(C\cap\p B_\r)e^{\f{t\r^2}4}d\r=&\int_0^{\sqrt{t}r}\mathcal{H}^{n-1}\left(C\cap\p B_{\f{s}{\sqrt{t}}}\right)e^{\f{s^2}4}\f{ds}{\sqrt{t}}\\
=&t^{-\f n2}\int_0^{\sqrt{t}r}\mathcal{H}^{n-1}(C\cap\p B_s)e^{\f{s^2}4}ds\\
\le&t^{-\f n2}\int_{B_{\sqrt{t}r}}\left|D\chi_{_{\sqrt{t}U}}(X)\right|e^{\f{|X|^2}4}dX\\
=&\int_{B_r}|D\chi_{_{U}}|e^{\f{t|X|^2}4}dX,
\endaligned
\end{equation}
where $U$ is any open set in $\R^{n+1}$ with $\p U\setminus B_r=C\setminus B_r$.
Letting $t\rightarrow0$ implies
\begin{equation}\aligned
\mathcal{H}^n(C\cap B_{r})\le \int_{B_r}|D\chi_{_{U}}|.
\endaligned
\end{equation}
Namely, $C$ is area-minimizing. Further, we suppose that a minimal cone $C$ is an $n$-dimensional \textbf{E}-minimizing self-expanding hypersurface in $\overline{\Om}$. Analog to the above argument, one can show that $C$ is an area-minimizing cone in $\overline{\Om}$. Hence we complete the proof.
\end{proof}


Furthermore, we have the following equivalence on multiplicity 1 area-minimizing cones and \textbf{E}-minimizing self-expanding hypersurfaces.
\begin{theorem}\label{equamcse}
Let $C$ be a regular minimal cone in $\R^{n+1}$ and $\Om$ be a domain with $\p\Om=C$, then $C$ is area-minimizing in $\overline{\Om}$
if and only if $C$ is an \textbf{E}-minimizing self-expanding hypersurface in $\overline{\Om}$.
\end{theorem}
\begin{proof}
In view of Lemma \ref{ammse}, we only need prove that if $C$ is area-minimizing in $\overline{\Om}$, $C$ would be an \textbf{E}-minimizing self-expanding hypersurface in $\overline{\Om}$.
Let $C$ be a regular, area-minimizing cone in $\overline{\Om}$. Assume that $T_r$ is a multiplicity one \textbf{E}-minimizing self-expanding current in $\overline{\Om}$ with $\mathrm{spt}T_r\setminus C\neq\emptyset$ and with $\p \left(\mathrm{spt}T_r\right)=C\cap\p B_r$ for each $r>0$.
By \cite{HaS}, there is a connected smooth embedded, star-shaped, area-minimizing hypersurface $S$ without boundary in $\Om$. By \eqref{dpErpBsC}, $r_1^{-1}\left(\mathrm{spt}T_{r_1}\right)\cap S=\emptyset$ for a sufficiently large constant $r_1>0$. There is a constant $r_2\in(0,r_1)$ such that $r_2^{-1}\left(\mathrm{spt}T_{r_1}\right)\cap S\neq\emptyset$ but $r^{-1}\left(\mathrm{spt}T_{r_1}\right)\cap S=\emptyset$ for $r>r_2$. Note that $r_2^{-1}\left(\mathrm{spt}T_{r_1}\right)$ has mean curvature
\begin{equation}\aligned\label{H**}
H=\f{1}2r_2^2X^N.
\endaligned
\end{equation}
If we see $r_2^{-1}\left(\mathrm{spt}T_{r_1}\right)$ being a hypersurface in the space $\left(\R^{n+1},e^{\f{r_2^2}{2n}|X|^2}\sum_{i=1}^{n+1}dx_i^2\right)$ with the induced metric, then the corresponding mean curvature is zero from \eqref{H**}.

Let $\Om_S$ be a domain containing $r_1^{-1}\left(\mathrm{spt}T_{r_1}\right)$ in $\R^{n+1}$ with boundary $\p\Om_S=S$. If we see $S$ being a hypersurface in the space $\left(\R^{n+1},e^{\f{r_2^2}{2n}|X|^2}\sum_{i=1}^{n+1}dx_i^2\right)$ with induced metric and mean curvature, then it has positive mean curvature pointing into $\R^{n+1}\setminus\Om_S$ in $\left(\R^{n+1},e^{\f{r_2^2}{2n}|X|^2}\sum_{i=1}^{n+1}dx_i^2\right)$ as $S$ is star-shaped and minimal in Euclidean space.  By Lemma \ref{mpbySolomon-W}, we get a contradiction. Hence, $\mathrm{spt}T_r=C$ for all $r>0$, and we complete the proof.
\end{proof}

\begin{lemma}\label{area-minimSE}
Let $C$ be an area-minimizing cone with an isolated singularity, then the support of any $n$-varifold self-expander $T$ with $\p T=0$ in $\R^{n+1}$ must be $C$
if the tangent cone of $T$ at infinity is $\llbracket C\rrbracket$.
\end{lemma}
\begin{proof}
Let $\Om_+$ and $\Om_-$ be two components divided by $C$. Let $S^\pm$ be smooth area-minimizing hypersurfaces in $\Om_\pm$,
which are connected smooth embedded hypersurfaces without boundary (The existence of $S^\pm$ are known in \cite{HaS}).
Let $T$ be a connected $n$-varifold self-expander with $\p T=0$ in $\R^{n+1}$ which converges to $\llbracket C\rrbracket$ in the varifold sense at infinity.
Then $\mathrm{spt}T$ is a graph over $C$ outside a compact set with the graphic function $f$.
If $\mathrm{spt}T\neq C$, then asymptotic behavior of $S^\pm$ and Lemma \ref{upgse} imply that there exists a constant $\ep>0$ such that $\mathrm{spt}(t T)\cap S^\pm=\emptyset$
for $t\in(0,\ep]$.

Denote $\g_{\pm}=\f{n-2}2\pm\sqrt{\f{(n-2)^2}4+\la_1}$, where $\la_1$ is the first (nonzero) eigenvalue of Jacobi operator of $C\cap\S^n$ in $\S^n$.
By \cite{HaS}, up to a compact set, $S^\pm$ is a graph over $C$ with the graphic functions $f_{\pm}$ satisfying
\begin{equation*}
\left\{\begin{split}
\mathrm{either}\ \ &f_{\pm}=(c_1^\pm+c_2^\pm\log r)r^{-\g_-}\varphi_1+O(r^{-\g_--\a})\qquad \mathrm{as}\ \  r\rightarrow\infty\\
\mathrm{or}\ \ \quad &f_{\pm}=c_1^\pm r^{-\g_+}\varphi_1+O(r^{-\g_+-\a})\qquad \mathrm{as}\ \  r\rightarrow\infty\\
\end{split}\right.
\end{equation*}
($\a>0$), where $(c_1^++c_2^+\log r)>0>(c_1^-+c_2^-\log r)$, and $c_2^\pm=0$ unless $\g_+=\g_-=\f{n-2}2$, and where $\pm c_1^\pm>0$ in the second identity.

Let $t_0=\sup\{t|\ \mathrm{spt}(sT)\cap(S^+\cup S^-)=\emptyset,\ s\in(0,t)\}$.
By Lemma \ref{mpbySolomon-W} with the space $N=\left(\R^{n+1},e^{\f{t_0^2|X|^2}{2n}}\sum_{i=1}^{n+1}dx_i^2\right)$, $\mathrm{spt}(t_0T)\cap(S^+\cup S^-)=\emptyset$.
Combining Lemma \ref{upgse}, there is a constant $r_1>0$ such that
\begin{equation}\aligned\label{6.666fpm}
\f12f^+(r)-sf\left(\f r{s}\right)\ge0\ge\f12f^-(r)-sf\left(\f r{s}\right)
\endaligned
\end{equation}
for all $t_0\le s\le 2t_0$ and $r\ge r_1$. In particular, $$\min\{|x-y|\big|\, x\in\mathrm{spt}(t_0T)\cap B_{r_1},\ y\in(S^+\cup S^-)\cap B_{r_1}\}>0.$$
With the help of \eqref{6.666fpm}, there is a constant $t_1\in(0,t_0)$ such that $\mathrm{spt}(tT)\cap S^\pm=\emptyset$ for $t\in(t_0,t_0+t_1)$.
This violates the choice of $t_0$.
Hence, we have $\mathrm{spt}(t T)\cap (S^+\cup S^-)=\emptyset$ for all $t\in(0,\infty)$. This is a contradiction clearly. We complete the proof.
\end{proof}

\subsection{Existence of \textbf{E}-minimizing self-expanding currents in mean convex cones}

\begin{lemma}\label{smsemc}
Let $C$ be a $C^2$-regular, mean convex cone with vertex at the origin in $\R^{n+1}$ pointing into a domain $\Om$ with $\p\Om=C$. If $T$ is a varifold self-expander with $\p T=0$ in $\overline{\Om}$, then either $\mathrm{spt}T\cap C=\emptyset$, or $\mathrm{spt}T=C$ and $C$ is minimal.
\end{lemma}
\begin{proof}
Obviously, $\mathrm{spt}T=C$ means that $C$ is minimal. In view of Lemma \ref{msem}, we suppose that $\mathrm{spt}T\neq C$ and $C$ is not a minimal cone, then obviously $\mathrm{spt}T\cap \mathrm{reg}C=\emptyset$ by Lemma \ref{mpbySolomon-W} with $N=\left(\R^{n+1},e^{\f{|X|^2}{2n}}\sum_{i=1}^{n+1}dx_i^2\right)$. So for showing $\mathrm{spt}T\cap C=\emptyset$, we only need prove $0\notin \mathrm{spt}T$.

Assume $0\in \mathrm{spt}T$, and let us show that this leads to a contradiction.
Let $\mathscr{T}_0$ be a tangent cone of $\mathrm{spt}T$ at $0$, then $\mathrm{spt}T\subset\overline{\Om}$ implies that $\mathscr{T}_0$ is a stationary cone in $\overline{\Om}$. According to Lemma \ref{mpbySolomon-W} with $N=\R^{n+1}$, $\mathscr{T}_0\cap C=\emptyset$,
and then $\mathscr{T}_0\subset\Om$.
Let $\Om_\a$ be the shape after a rotation of $\Om$ by the action $\a$ in orthogonal group $O(n):\S^n\rightarrow\S^n$. Obviously, there is an $\a_0\in O(n)$ such that $\mathscr{T}_0\subset\overline{\Om}_{\a_0}$ and $\mathscr{T}_0\cap\p\Om_{\a_0}\neq\emptyset$. However, this violates Lemma \ref{mpbySolomon-W} with $N=\R^{n+1}$. Hence $0\notin\mathrm{spt}T$, and we complete the proof.
\end{proof}

\begin{theorem}\label{WESEmcC}
Let $C$ be a $C^2$-regular, mean convex but not area-minimizing cone pointing into $\Om$ in $\R^{n+1}$ with $\p\Om=C$, then there is a multiplicity one \textbf{E}-minimizing self-expanding current $T$ in $\Om$ with $\p T=0$ such that the tangent cone of $T$ at infinity is $\llbracket C\rrbracket$ .
\end{theorem}
\begin{proof}
Note that $C\cap\S^n$ is $C^2$-continuous by the assumption of mean convex $C$. Then there is a uniform positive constant $0<\k<\f12$ depending only on $C$ so that for any $y\in C\cap\S^n$, we have a cone $C_{\k}$ such that $C_{\k}\cap C\cap\S^n=y$ and $C_{\k}$ is $\left\{\left(x,\f{|x|}\k\right)\big| x\in\R^n\right\}$ up to a rotation.

By compactness theorem for currents (see \cite{LY}\cite{S} for instance),
there exists an open set $E_r\subset\Om$ for any $r>0$ such that $\p  E_r\setminus B_r=C\setminus B_r$ and
\begin{equation}\aligned
\int_{B_r}|D\chi_{_{E_r}}|e^{\f{|X|^2}4}dX=\inf\left\{\int_{B_r}|D\chi_{_K}|e^{\f{|X|^2}4}dX\Big|\ \p K\setminus B_r=C\setminus B_r,\ K\subset\overline{\Om}\right\}.
\endaligned
\end{equation}
Hence $\p \llbracket E_r\rrbracket\llcorner B_r$ is an $n$-dimensional multiplicity one \textbf{E}-minimizing self-expanding current in $\overline{\Om}$. By Lemma \ref{smsemc}, either $\p E_r\cap B_r\cap C=\emptyset$ or $\p E_r=C$.
Combining $E_r\subset \overline{\Om}$ and the proof of Theorem \ref{WeakexistSE}, there is a constant $c_\k>0$ such that
$$d_H(\p E_r\cap\p B_s,\ C\cap\p B_s)\le \f{2(n+2)}{\k^2s}$$
for every $s\in[c_\k,r]$. Following the proof of Theorem \ref{WeakexistSE}, there is a sequence $r_i\rightarrow\infty$ such that $\p \llbracket E_{r_i}\rrbracket\llcorner B_{r_i}$ converges to an $n$-dimensional multiplicity one \textbf{E}-minimizing self-expanding current $T$ in $\overline{\Om}$ with $\p T=0$ and $\lim_{r\rightarrow\infty}\f1rT=\llbracket C\rrbracket$. By Lemma \ref{ammse} and Lemma \ref{smsemc}, $\mathrm{spt}T\subset\Om$ since $C$ is not area-minimizing.
\end{proof}

\subsection{Existence and uniqueness of smooth $\mathbf{E}$-minimizing self-expanding hypersurfaces}
For a domain $U_0$, and a set $E_a\subset U_0$, we say a set $E_b$ \emph{on one side of $E_a$ in $U_0$} if $E_b$ is contained in the closure of one of the connected components of $U_0\setminus E_a$ in the current paper. When $U_0$ is a Euclidean space, we sometimes omit $U_0$.
\begin{lemma}\label{sesem}
Let $E_1$ and $E_2$ be two open sets satisfying $0\notin\overline{E_1\cap U}$ and $E_2\cap U\subset E_1\cap U$ for some open set $U\subset\R^{n+1}$. Let $T_1=\p\llbracket E_1\rrbracket\llcorner U$, $T_2=\p\llbracket E_2\rrbracket\llcorner U$, and $\sqrt{\g}|T_1|$ and $|T_2|$ be both $n$-dimensional varifold self-expanders in $U$ for some constant $\g>1$ with $\mathrm{spt}(t T_1)\cap\mathrm{spt}T_2\cap U=\emptyset$ for each $t\in(0,1)$. If either $\sqrt{\g}T_1$ or $T_2$ is an $\mathbf{E}$-minimizing self-expanding current in $U$, then either $T_1\llcorner U=T_2\llcorner U$ or $\mathrm{spt}T_1\cap\mathrm{spt}T_2\cap U=\emptyset$.
\end{lemma}
\begin{proof}
We suppose that $\mathrm{spt}T_1\cap\mathrm{spt}T_2\cap U\neq\emptyset$ and there is a point $x_0\in\mathrm{spt}T_1\cap\mathrm{spt}T_2\cap U$. If $x_0\in\mathrm{reg}T_1$, then spt$T_1$ is smooth in a neighborhood of $x_0$ by regularity of minimal hypersurfaces.
Denote $\Si_t=\sqrt{t-1+1/\g}\,\mathrm{spt}(\sqrt{\g}T_1)$ for $t>1-\f1\g$.
By the definition of $T_1,T_2$, both $\Si_t$ and $\sqrt{t}\,\mathrm{spt}T_2$ are mean curvature flows for the time $t\in(1-\f1\g,1+\f1\g)$.
Noting that
$$\f{\sqrt{\g}}{\sqrt{t}}\sqrt{t-1+1/\g}=\sqrt{\g}\sqrt{1-(1-1/\g)/t}<1\quad \mathrm{for\ every}\ t\in(1-1/\g,1).$$
From $\mathrm{spt}(t T_1)\cap\mathrm{spt}T_2\cap U=\emptyset$ for each $t\in(0,1)$ and $x_0\in\mathrm{reg}T_1\cap\mathrm{spt}T_2\cap U$, we get $\Si_t\cap\sqrt{t}\,\mathrm{spt}T_2\cap U=\emptyset$ for any $t\in(1-1/\g,1)$ and $x_0\in \mathrm{reg}\Si_1\cap\mathrm{spt}T_2\cap U$. For every $t\in(1-\ep,1]$ with the suitably small $\ep>0$, both $\Si_t\cap B_\ep(x_0)$ and $(\sqrt{t}\,\mathrm{spt}T_2)\cap B_\ep(x_0)$ can be written as smooth graphs over small neighborhoods of $x_0$ in the tangent plane $T_{x_0}\Si_1$.
Then $x_0\in \Si_1\cap\mathrm{spt}T_2\cap U$ violates the strong maximum principle of parabolic equations.


Therefore, we only need to consider the case: $x_0\in\mathrm{Sing}T_1\cap\mathrm{Sing}T_2\cap U$. Assume $\mathrm{reg}T_1\cap\mathrm{reg}T_2=\emptyset$ or else we complete the proof. Obviously, every tangent cone of $T_1$ or $T_2$ at $x_0$ is an $n$-dimensional area-minimizing cone in $\R^{n+1}$. Hence for any sequence $\la_j\rightarrow0$ there is a subsequence $\la_{j'}$ of $\la_j$ such that $\la_{j'}^{-1}(|T_1|-x_0)$ and $\la_{j'}^{-1}(|T_2|-x_0)$ both converge to a cone $C_T$ in $\R^{n+1}$ in the varifold sense, or else this contradicts Lemma \ref{msem} (see also \cite{S2}). Since either $\sqrt{\g}\p\llbracket E_1\rrbracket$ or $\p\llbracket E_2\rrbracket$ is an $\mathbf{E}$-minimizing current in $U$, then $C_T$ is a multiplicity one area-minimizing cone (see Theorem 37.2 in \cite{S} for example). In particular, $C_T$ is not a hyperplane by Allard's regularity theorem. For any constant $s>0$, let $t_{j'}=1-\la_{j'}s$ for the sufficiently large $j'$. Note $x_0\neq0$ from $0\notin E_2\cap U\subset E_1\cap U$. By
$$\la_{j'}^{-1}(t_{j'}\p E_1-x_0)=\la_{j'}^{-1}\left(t_{j'}(\p E_1-x_0)-(1-t_{j'})x_0\right)=\left(\la_{j'}^{-1}-s\right)(\p E_1-x_0)-sx_0,$$
we conclude that $\la_{j'}^{-1}\llbracket t_{j'}(U\cap\p E_1)-x_0\rrbracket$ converges to $C_{T,s}\triangleq C_T-sx_0$ with the vertex at $-sx_0$ as $j'\rightarrow\infty$ in the varifold sense. Since $t_{j'}(\mathrm{spt}T_1)$ is on one side of $\mathrm{spt}T_2$ in $U$ by the assumption, then $\la_{j'}^{-1}(t_{j'}\mathrm{spt}T_1-x_0)$ is on one side of $\la_{j'}^{-1}(\mathrm{spt}T_2-x_0)$. Taking the limit, we conclude that
$\mathrm{spt}C_{T,s}$ is also on one side of $\mathrm{spt}C_T$ for any $s>0$. Recall $C_{T,s}$ and $C_T$ are both minimizing cones with different vertices, then by Lemma \ref{msem}, $\mathrm{spt}C_T\cap \mathrm{spt}C_{T,s}=\emptyset$ for each $s>0$. Hence, $\mathrm{spt}C_T$ can be written as a graph over a hyperplane $P_s$, which is perpendicular to the vector $\overrightarrow{ox_s}$.
However, this is impossible by blowing up argument (at singular sets of spt$C_T$) and the proof of Bernstein theorem. Hence we complete the proof.
\end{proof}

Theorem \ref{WESEmcC} gives an existence for varifold self-expanders with prescribed mean convex cones at infinity. Fortunately, we are able to point out that such varifold self-expanders are smooth actually. Firstly, we focus on regular minimal but not minimizing cones.
\begin{theorem}\label{Minimalfoliation}
Let $C$ be an $n$-dimensional regular minimal but not minimizing cone in $\R^{n+1}$, and $\Om$ be a domain with $\p\Om=C$, then there is a unique smooth complete embedded \textbf{E}-minimizing self-expanding hypersurface $M$ with $\lim_{r\rightarrow\infty}r^{-1}M=C$ in $\Om$. Moreover, $M$ has positive mean curvature everywhere and $\{\sqrt{t}M\}_{t>0}$ is a foliation of $\Om$.
\end{theorem}
\textbf{Remark.} For a smooth hypersurface $S$ and a regular cone $C$ in $\R^{n+1}$, we say that $\lim_{r\rightarrow\infty}r^{-1}S=C$ (or the tangent cone of $S$ at infinity is $C$) if $\lim_{r\rightarrow\infty}r^{-1}\llbracket S\rrbracket=\llbracket C\rrbracket$. From \eqref{monSEr-nBr} and the argument in Theorem 4.1 of \cite{CD}, $S$ has Euclidean volume growth and it is proper if $\lim_{r\rightarrow\infty}r^{-1}\llbracket S\rrbracket=\llbracket C\rrbracket$.
\begin{proof}
By Theorem \ref{equamcse} and Theorem \ref{WESEmcC}, there is a multiplicity one \textbf{E}-minimizing self-expanding current $T$ with $\p T=0$ in $\Om$ with tangent cone $C$ at infinity, where $\Om$ is a component of $\R^{n+1}\setminus C$. Let $M=\mathrm{spt}T$.
From the discussion in section 5,
up to a compact set, $tM$ also can be represented as a graph over $C\setminus B_{tr_0}$ with the graphic function $w_t$ defined by
$$w_t(x)=tw\left(\f xt\right)>0$$
for $x\in C\setminus B_{tr_0}$ and $t>0$.

By asymptotic behavior \eqref{asymse}, there is a constant $\de_0>0$ such that $\sqrt{\de_0}\ M\cap M=\emptyset$.
Let $\de_1=\sup\{\de| \sqrt{s}M\cap M=\emptyset,\ s\in(0,\de)\}$. Then $\de_1\le1$ clearly.
By Lemma \ref{sesem}, we conclude that $\sqrt{\de_1}M\cap M=\emptyset$ provided $\de_1<1$. From Corollary \ref{updownC}, there are constants $r_1>r_0$ and $\ep_5>0$ such that
\begin{equation}\aligned
w(x)\ge\ep_5r^{-n-1}e^{-\f{r^2}{4}}>\ep_5^{-1}t\left(\f rt\right)^{-n-1}e^{-\f{r^2}{4t^2}}\ge w_t(x)
\endaligned
\end{equation}
for all $x$ with $|x|=r\ge r_1$ and any $\de_1\le t\le\f{1+\de_1}2$. Hence there is a constant $\de_2\in(\de_1,\f{1+\de_1}2)$ such that $\sqrt{s}M\cap M=\emptyset$ for $s\in(0,\de_2]$. This violates the maximum of $\de_1$. Therefore, $\sqrt{s}M\cap M=\emptyset$ for all $0<s<1$ and $\{\sqrt{s}M\}_{s>0}$ is a foliation of $\Om$. So $M$ is embedded and each ray $\{\la\xi\in\R^{n+1}|\ \la>0\}$, corresponding to $\xi\in\Om\cap\p B_1$, intersects $M$ in a single point $y$. Then for a suitable $\r_y>0$ we can write $M\cap B_{\r_y}(y)$ as a graph of some function $w$ with bounded gradient over a hyperplane $P$ normal to $\xi$ (see also the proof of Theorem 2.1 in \cite{HaS}). Since $M=\mathrm{spt}T$ is an \textbf{E}-minimizing self-expanding hypersurface and $\mathcal{H}^{s}(\mathrm{sing}T)=0$ for $s>\max\{0,n-7\}$, it follows that sing$T\cap B_{\r_y}(y)=\emptyset$. Thus sing$T=\emptyset$ and $M$ is smooth without boundary. From the foliation $\{\sqrt{s}M\}_{s>0}$, $M$ has nonnegative mean curvature pointing into the domain $\Om$. \eqref{meancurvature} and Hopf maximum principle imply that $M$ has positive mean curvature everywhere.

Let $S$ be an $n$-dimensional smooth complete self-expander in $\Om$ with $\lim_{t\rightarrow0}tS=C$. Then there is a constant $\ep>0$ such that $\sqrt{\ep}(\mathrm{spt}S)\cap M=\emptyset$. By the above argument and Lemma \ref{sesem}, we obtain $s(\mathrm{spt}S)\cap M=\emptyset$ for all $0<s<1$. Similarly, $sM\cap \mathrm{spt}S=\emptyset$ for all $0<s<1$. Hence $\mathrm{spt}S=M$ and the uniqueness holds.
\end{proof}

Secondly, we utilize the asymptotic analysis in Lemma \ref{AsmPMC} to get the smooth existence for every prescribed regular cone with positive mean curvature at infinity.
\begin{corollary}\label{posiSE}
Let $C$ be an $n$-dimensional $C^{3,\a}$-regular cone in $\R^{n+1}$ which has positive mean curvature pointing into the domain $\Om$ with $\p\Om=C$, then there is a unique smooth complete embedded \textbf{E}-minimizing self-expanding hypersurface $M$ with $\lim_{r\rightarrow\infty}r^{-1}M=C$ in $\Om$. Moreover, $M$ has positive mean curvature everywhere and $\{\sqrt{t}M\}_{t>0}$ is a foliation of $\Om$.
\end{corollary}
\begin{proof}
The proof is similar to the proof of Theorem \ref{Minimalfoliation}, where the asymptotic behavior of self-expanders is described in Lemma \ref{AsmPMC}. So we omit it.
\end{proof}

Combining Theorem \ref{Minimalfoliation} and Corollary \ref{posiSE}, we eventually deduce the smooth existence for mean convex case.
\begin{theorem}\label{mainTh}
Let $C$ be an $n$-dimensional $C^{3,\a}$-regular, mean convex but not minimizing cone pointing into a domain $\Om$ with $\p\Om=C$ in $\R^{n+1}$, then there is a unique smooth complete embedded \textbf{E}-minimizing self-expanding hypersurface $M$ in $\Om$ with tangent cone $C$ at infinity. Moreover, $\{\sqrt{t}M\}_{t>0}$ is a foliation of $\Om$.
\end{theorem}
\begin{proof}
Let $C=C\Si$ be a $C^{3,\a}$-regular cone over a mean convex hypersurface $\Si$ in $\S^n$. We assume that $C$ is not a minimal cone, or else we have finished the proof from Theorem \ref{Minimalfoliation}. Then there is a constant $\de>0$ such that $t\in[0,\de]\rightarrow\Si_t$ is a smooth mean curvature flow starting from $\Si$.
Obviously, $\Si_t$ is embedded and has positive mean curvature, so there is a smooth complete embedded \textbf{E}-minimizing self-expanding hypersurface $M_t$ in $\Om$ for $t>0$, which has positive mean curvature everywhere and $\lim_{r\rightarrow\infty}\f1rM_t=C\Si_t$.
There exists a sequence $0<t_i\rightarrow0$ such that $\llbracket M_{t_i}\rrbracket\rightharpoonup T$ in the varifold sense. Then $T$ is a multiplicity one \textbf{E}-minimizing self-expanding current in $\overline{\Om}$ with $\p T=0$. From \eqref{dTsBsC}, we get $\lim_{r\rightarrow\infty}\f1rT=\llbracket C\Si\rrbracket$.
Denote $M=\mathrm{spt}T$. By Allard's regularity theorem, there is a constant $r_2>1$ such that $M\setminus B_{r_2-1}$ is smooth with bounded geometry.
Hence $M_{t_i}\setminus B_{r_2-1}$ converges to $M\setminus B_{r_2-1}$ smoothly, which implies that the mean curvature $H$ is nonnegative in $M\setminus B_{r_2}$.
Hence the mean curvature $H$ is positive in $M\setminus B_{r_2}$ by maximum principle for \eqref{meancurvature}. Denote $\ep_6\triangleq\f12\inf_{\p B_{r_2}}H>0$. There exists a constant $\de_3\in(0,\de]$ such that the mean curvature $H_{t_i}$ of $M_{t_i}$ satisfies $H_{t_i}\ge\ep_6$ on $M_{t_i}\cap\p B_{r_2}$ for $t_i\in(0,\de_3)$. By maximum principle for \eqref{meancurvature} again, we get
$$\f12\lan X,\nu_{t_i}\ran=H_{t_i}\ge\ep_6\qquad \mathrm{on}\quad M_{t_i}\cap B_{r_2},$$
where $\nu_{t_i}$ is the unit normal vector of $M_{t_i}$.
Since $M_{t_i}\cap B_{r_2}$ is a graph over a subset of $\Om\cap\S^n$, then from the above inequality, $M\cap B_{r_2}$ is still a graph over a subset of $\overline{\Om}\cap\S^n$. Hence $M\cap B_{r_2}$ is smooth, and $M$ is a smooth complete embedded \textbf{E}-minimizing self-expanding hypersurface in $\Om$. Moreover, $M$ is mean convex pointing into $\Om$ as $M_{t_i}\rightharpoonup M$, then $M$ has positive mean curvature pointing into $\Om$ by \eqref{meancurvature} and maximum principle. Since $sM_{t_i}$ is on one side of $M_{t_i}$ for any $s\neq1$, then $sM$ is on one side of $M$. Maximum principle implies that $sM\cap M=\emptyset$ for $s\neq1$. So $\{\sqrt{t}M\}_{t>0}$ is a foliation of $\Om$.

Suppose that $M'$ is another smooth complete self-expander in $\Om$ with tangent cone $C$ at infinity.
Set
$$\de_{4,i}=\sup\left\{s\in(0,1]\Big|\ \sqrt{t}M'\cap M_{t_i}=\emptyset,\ 0<t<s\right\}.$$
If $\sqrt{\de_{4,i}}M'\cap M_{t_i}\neq\emptyset$ and $\sqrt{t}M'\cap M_{t_i}=\emptyset$ for $0<t<\de_{4,i}$, this  contradicts Lemma \ref{sesem}. Then $\sqrt{t}M'\cap M_{t_i}=\emptyset$ for $0<t\le\de_{4,i}$. Since $\lim_{r\rightarrow0}r^{-1}M_{t_i}=C\Si_{t_i}$ and $\Si_{t_i}\cap\Si=\emptyset$, then $\de_{4,i}$ attains its maximum, namely, $\de_{4,i}=1$. Hence $\bigcup_{0<t\le1}\sqrt{t}M'$ is on one side of $M_{t_i}$. Letting $t_i\rightarrow0$ implies that $\bigcup_{0<t\le1}\sqrt{t}M'$ is on one side of $M$. Since $M$ is a smooth self-expander, then either $M=M'$ or $M\cap M'=\emptyset$.

Assume $M\cap M'=\emptyset$. Now we deduce the contradiction.
There exists a constant $r_3>0$ such that outside a compact set, $M'$ can be represented as a graph over $\sqrt{t}M\setminus r_3$ with the graphic function $w_t$.
Denote $\pi_t$ be the projection from $C$ to $\sqrt{t}M$. Hence for any fixed $r\ge r_3$, for the sufficiently small $t>0$ we have $w_t\circ\pi_t(\xi r)>0$.
From Lemma \ref{upgse}, for the sufficiently large $r_3>0$, there is a constant $t_{r_3}$ such that for each $0<t\le t_{r_3}$ we have
$$w_t\circ\pi_t(x)>\ep_{7,t}|x|^{-n-1}e^{-\f{|x|^2}4}$$
on $\sqrt{t}M\setminus r_3$ for some sufficiently small constant $\ep_{7,t}>0$ depending on $t$.

Set
$$\de_5=\sup\{s\in(0,1)|\ \sqrt{t}M\cap M'=\emptyset,\ 0<t<s\}.$$
So we have $\sqrt{t}M\cap M'=\emptyset$ for $t\in(0,\de_5)$. We assume $\de_5<1$. Taking $t\rightarrow\de_5$, for any $r>0$ we get $\sqrt{\de_5}M\cap M'\cap B_r=\emptyset$ by using
Lemma \ref{sesem} in the open set $B_{2r}$. Namely, $\sqrt{\de_5}M\cap M'=\emptyset$ and $\sqrt{\de_5}M$ is on one side of $M'$. From Lemma \ref{upgse} again, there is a constant $\de_5<\de_6<1$ such that $\sqrt{\de_6}M\cap M'=\emptyset$, which contradicts to the choice of $\de_5$.
Therefore we get $\de_5=1$. Due to the fact that $\bigcup_{0<t\le1}\sqrt{t}M'$ is on one side of $M$, we deduce $M'=M$, and then complete the proof.
\end{proof}

{\bf Acknowledgement.}
The author would like to express his gratitude to Prof. Y. L. Xin for inspiring discussion and valuable advice, and to Prof. J. Jost for his constant encouragement and support. He would also like to thank Prof. F. Morgan for his interests and suggestions.
The author would like to express his sincere gratitude to the referees for valuable comments that will help to improve the quality of the manuscript.
The author is supported partially by Natural Science Foundation of Shanghai (Grant No. 15ZR1402200) and
Shanghai Municipal Education Commission.

\bibliographystyle{amsplain}

\end{document}